\documentclass[12pt]{amsart} 
 
\usepackage{amsmath} 
\usepackage{amsthm} 
\usepackage{amssymb} 
\usepackage{fancyhdr} 
 
\usepackage{color} 
 
\makeatletter 
  
  \@addtoreset{equation}{section} 
 \makeatother

\title{The Alexandrov-Toponogov comparison theorem for radial curvature} 
 
\author{Nobuhiro Innami} 
\author{Katsuhiro Shiohama} 
\author{Yuya Uneme}
 
\address[N. Innami]{Department of Mathematics, Faculty of Science, Niigata 
University, Niigata, 950-2181, JAPAN} 
\address[K. Shiohama]{Department of Applied Mathematics, Faculty of Sciences, 
Fukuoka University, 8--19--1, Nanakuma, Jonan-ku, Fukuoka, 814-0180, JAPAN} 

\address[Y. Uneme]{Graduate School of Science and Technology, Niigata University, Niigata, 950-2181, JAPAN}

\thanks{Research of the first author was partially supported by Grant-in-Aid for Scientific Research (C), 22540072.}
\thanks{Research of the second author was partially supported by Grant-in-Aid for Scientific Research (C) 22540106}

\date{\today} 
\theoremstyle{plain} 
\newtheorem{thm}{Theorem} 
\newtheorem{prop}[thm]{Proposition} 
\newtheorem{lem}[thm]{Lemma} 
\newtheorem{cor}[thm]{Corollary} 
\newtheorem{exa}[thm]{Example} 
\newtheorem{rem}[thm]{Remark} 
\newtheorem{ass}[thm]{Assertion}
 
\begin{document} 
\maketitle 

\renewcommand{\thefootnote}{}

\begin{abstract} 
We discuss the Alexandrov-Toponogov comparison theorem under the conditions of 
radial curvature of a pointed manifold $(M,o)$ with reference surface of 
revolution $(\widetilde M, \tilde o)$. There are two obstructions to make the 
comparison theorem for a triangle one of whose vertices is a base point $o$. One 
is the cut points of another vertex $\tilde p \not=\tilde o$ of a comparison 
triangle in $\widetilde M$. The other is the cut points of the base point $o$ in 
$M$. We find a condition under which the comparison theorem is valid for any 
geodesic triangle with a vertex at $o$ in $M$. 
\end{abstract} 
 
\section{Introduction} 
The Alexandrov-Toponogov comparison theorem (shortly ATCT) has been very useful 
in the study of geometry of geodesics including Riemannian geometry. In the 
present paper we discuss ATCT under the conditions of radial curvature of a 
pointed manifold $(M,o)$. We can see historical remarks and many references 
about ATCT in \cite{IMS-0}, \cite{KT2}. 
 
Reference space is a surface of revolution $(\widetilde M, \tilde o)$ with a 
geodesic polar coordinate system $(r, \theta)$ around $\tilde o$ such that its 
metric is given by 
\[ 
ds^2 = dr^2 + f(r)^2d\theta^2. 
\] 
Let $p \in M$ be fixed and set $\tilde p=(d(o,p),0) \in \widetilde M$. We may
find a point $\tilde q \in \widetilde M$ satisfying 
\[ 
d(\tilde o, \tilde q)=d(o,q) \quad\quad\quad \mbox{and} \quad\quad\quad d(\tilde 
p, \tilde q)=d(p,q) 
\] 
for any point $q \in M$. We call $\tilde q$ the {\it reference point} of $q$ and 
the map $\Phi : M \rightarrow \widetilde M$ given by $q \mapsto \tilde 
q$ the {\it reference map} of $M$ into $\widetilde M$. 
It is not certain whether or not every point $q \in M$ has a reference point and every geodesic triangle $\triangle opq$, $p, q 
\in M$, admits the corresponding geodesic triangle $\triangle \tilde o \tilde p 
\tilde q$, $\tilde p, \tilde q \in \widetilde M$.

There are two obstructions to ATCT. In the present paper we show that ATCT is established if the obstructions do not occur simultaneously at any point in $M$.

Let $T(p,q)$ denote a minimizing geodesic segment connecting $p$ and 
$q$ in $M$ and $\widetilde T(p,q)=\Phi(T(p,q))$. 
ATCT is valid if the reference curve $\widetilde T(p,q)$ and the minimizing geodesic segment $T(\tilde p, \tilde q)$ satisfy the good positional relation (see (2.2)). 
One obstruction to establish it is the appearance of cut points $Cut(\tilde p)$ to $\tilde 
p$ in the reference curve $\widetilde T(p, q)$ of $\triangle o p q$. In fact, we do not know what relation there is between the positions of the reference curve 
$\widetilde T(p,q)$ and $T(\tilde p, \tilde q)$ if the reference curve 
$\widetilde T(p,q)$ intersects $Cut(\tilde p)$. 
The cases treated in all papers for ATCT referred to a surface of revolution are free from the obstruction on the cut loci. 
Actually, it is proved in \cite{IMS} 
that ATCT for any geodesic triangle with a vertex $o$ holds if there are no cut 
points in a half part of $\widetilde M$ whose boundary consists of two meridians 
with angle $\pi $ at the vertex $\tilde o$.

In general, the composition of the reference map $\Phi$ with $r$-coordinate function may have a local maximum. This may cause a bad positional relation between $\widetilde T(p,q)$ and $T(\tilde p, \tilde q)$.
The other obstraction is the existence of the local maximum points of the 
distance function to $o$ in $M$, which is restricted to an ellipsoid with foci 
at $o$ and $p$. Let $E(p) \subset M$ be the set of all such local maximum points. Then, $E(p)$ 
is a subset of the cut points to the base point $o$ (see Lemma \ref{notmaximum}). The 
reference points of $E(p)$ may be in the boundary of the image  of the reference 
map of $M$ locally, so there possibly exists a minimizing geodesic segment whose 
endpoints are reference points but containing a non-reference point in 
$\widetilde M$. This situation shows that ATCT may not be true. 

It is natural to ask what condition verifies the existence of a minimizing geodesic segment $T(\tilde p, \tilde q)$ in $\widetilde M$ which has the good positional relation with $\widetilde T(p,q)$. 
We assume that the reference points of all points in $E(p)$ are not cut points 
of $\tilde p$ (see (2.1)). Under this assumption we prove that ATCT holds for any geodesic 
triangle $\triangle opq$.

We say that $(M,o)$ is {\it referred} to $(\widetilde M, 
\tilde o)$ if every radial sectional curvature at $p \in M$ is bounded below by 
$K(d(o,p))$, where $K(t)$ is the Gauss curvature of $\widetilde M$ at the points in the parallel $t$-circle. 
 
\begin{thm}[The Alexandrov convexity]\label{main1}
Assume that a complete pointed Riemannian manifold $(M,o)$ is referred to a 
surface of revolution $(\widetilde M, \tilde o)$. If a point $p \in M$ satisfies 
$ 
\Phi(E(p))\cap Cut(\tilde p)=\emptyset, 
$ 
then for any geodesic triangle $\triangle opq \subset M$ with $q \not= o, p$ there exists its comparison triangle $\triangle \tilde o \tilde p \tilde q \subset \widetilde M$ such that
$d(\tilde o , \tilde x) \le d(o, x)$ for any point $\tilde x \in T(\tilde p, \tilde q)$ and $x \in T(p,q)$ with $d(\tilde p, \tilde x)=d(p,x)$. Here $T(p,q)$ and $T(\tilde p, \tilde q)$ are the bases of geodesic triangles $\triangle opq$ and $\triangle \tilde o \tilde p \tilde q$, respectively.
\end{thm} 
 
\begin{thm}[The Toponogov comparison theorem]\label{main2}
Under the same assumptions as in Theorem 1, every geodesic triangle $\triangle 
opq \subset M$ admits its comparison triangle $\triangle \tilde o \tilde p 
\tilde q \subset \widetilde M$ such that 
\begin{equation*} 
\angle opq \ge \angle \tilde o \tilde p \tilde q, \quad\quad 
\angle oqp \ge \angle \tilde o \tilde q \tilde p, \quad\quad 
\angle poq \ge \angle \tilde p \tilde o \tilde q. 
\end{equation*} 
\end{thm}

As an application of these theorems we are allowed to define a pointed Alexandrov space $(M,o)$ with radial curvature bounded below by a function $K$ (see Remark \ref{pole} and sequent paragraphs).

\begin{cor}\label{App-1}
Let $(M,o)$ be a compact Riemannian manifold which is an Alexandrov space with a base point at $o$ with radial 
curvature bounded below by the function $K$. Here $K : [0,\ell] 
\rightarrow \mathbb{R}$, $\ell < \infty$, is the radial curvature function of $(\widetilde M, \tilde o)$. Then, the perimeters of all geodesic triangles $\triangle opq$ in $M$ are less than or equal to $2\ell$ and the diameter of $M$ is less than or equal to $\ell$.
Moreover, if there exists a geodesic triangle $\triangle opq$ in $M$ whose perimeter is $2\ell$, then $M$ is isometric to the warped product manifold whose warping function is $K$.
In particular, the same conclusion holds for $M$ if the diameter of $M$ is $\ell$.
\end{cor}

The following corollary is proved in \cite{KO} when $M$ is a noncompact pointed 
 Riemannian manifold with radial curvature bounded below by the function $K$ which is monotone non-increasing. We call such a surface of revolution with monotone non-increasing curvature function a {\em von Mangoldt surface}.
There are no cut points in an open half part of a von Mangoldt surface $\widetilde M$ whose boundary consists of two meridians with angle $\pi $ at the vertex $\tilde o$ (see \cite{T}).

\begin{cor}\label{App1}
Let $(M,o)$ be a noncompact Alexandrov space with a base point at $o$ with radial 
curvature bounded below by the function $K$. Here $K : [0,\infty ) 
\rightarrow \mathbb{R}$ is the radial curvature function of $(\widetilde M, 
\tilde o)$. If the total curvature of $\widetilde M$ is positive, then $M$ has 
one end and has no straight line. 
\end{cor} 
 
These results will be stated more precisely in \S 2 after introducing some 
definitions and notations. 

The idea of the proof of the theorems is this. 
The good positional relation (2.2) is equivalent to the Alexsandrov convexity and the Toponogov angle comparison (see Remark \ref{ordinary}).  
Therefore, we study what positional relation holds between the reference curve $\widetilde T(p,q)$ of every  minimizing geodesic segment $T(p,q)$ in $M$ and a minimizing geodesic segment $T(\tilde p, \tilde q)$ in $\widetilde M$. 
To do this we use the partial order $\le$ in the set of all curves which are parameterized by the angle coordinate $\theta$ in $\widetilde M$.
Let $U(\tilde p, \tilde q)$ and $L(\tilde p, \tilde q)$ denote the minimizing geodesic segments connecting $\tilde p$ and $\tilde q$ in $\widetilde M$ such that $L(\tilde p, \tilde q) \le T(\tilde p, \tilde q) \le U(\tilde p, \tilde q)$ for any minimizing geodesic segment $T(\tilde p, \tilde q)$, namely all minimizing geodesic segments $T(\tilde p, \tilde q)$ lie in the biangle domain bounded by $L(\tilde p, \tilde q)\cup U(\tilde p, \tilde q)$ in $\widetilde M$ (see \S 6). 

Let $\widetilde M_{\tilde p}^+$ denote the half part of $\widetilde M$ bounded by the union of the meridians through $\tilde p$ and opposite to $\tilde p$.
For $r > d(o,p)$ we define 
an ellipsoid in $M$ by $E(o,p \, ;r)=\{ x \in M \, | \, d(o,x)+d(p,x)=r \}$. 
Let $r_0$ be the least upper bound of the set of all $r_1 > d(o,p)$ 
satisfying the following properties: If $q \in E(o,p \, ;r)$ for an $r \in (d(o,p), r_1)$, then
\begin{enumerate} 
\item[(C1)] there exists a minimizing geodesic segment $T(p, q)$ such that $T(p,q)$ is contained in the set $\Phi^{-1}(\widetilde M_{\tilde p}^+)$ and $\widetilde T(p,q) \ge U(\tilde p, \tilde q) $, 
\item[(C2)] every minimizing geodesic segment $T(p,q)$ is contained in the set $\Phi^{-1}(\widetilde M_{\tilde p}^+)$ and satisfies
$\widetilde T(p,q) \ge L(\tilde p, \tilde q) $.
\end{enumerate} 

It follows from \cite{IMS-0}, \cite{IMS} and  \cite{KT2} that $r_0 > d(o,p)$ (see Lemma \ref{fundamental}). We then prove that every geodesic triangle $\triangle opq$ in $M$ for $q \in E(o, p \, ; r)$, $d(o,p) < r  < r_0$, has a comparison triangle $\triangle \tilde o \tilde p \tilde q$ in $\widetilde M$ satisfying (2.3) (Assertion \ref{Ass1}).
All points in $E(o,p \, ;r_0)$ satisfy (C1) and (C2) again (Assertion \ref{Ass2}). 
We see, from the assumption of our theorems, that even when $q \in E(o,p \, ;r_0)$ with $q \not\in Cut(p)$ and $\tilde q \in Cut (\tilde p)$, the reference curve $\widetilde T_e(p,q)$ of the maximal minimizing geodesic $T_e(p, q)$ through $p$ and $q$ crosses
 $Cut(\tilde p)$ from the far side to the near side from $\tilde o$ in $\widetilde M$.
This fact shows that there exists an $r' > r_0$ such that (C1) and (C2) are true for any $r$, $d(o,p) < r < r'$ (see Assertions \ref{Ass3}, \ref{Ass4} and \ref{Ass5}). This means that the domain bounded by $E(o, p \, ; r_0)$ covers $M$.
 
The rest of this article is organized as follows. In \S 2 we state our results precisely with giving some notions we need. 
In \S 3 we give some properties of circles and ellipses in a surface of revolution. 
Moreover we give a sufficient condition for a point $q \in M$ not being contained in $E(p)$, and an example showing the property of ellipses which is essentially different from circles.
From \S 4 we start studying reference curves. 
In \S 4 we give the fundamental properties of the reference curves and the reference reverse curves. 
In \S 5 we treat the case that the reference curves from $\tilde p$ do not meet $Cut (\tilde p)$ in $\widetilde M$. 
In \S 6 we study the reference curves from $\tilde p$ meeting $Cut (\tilde p)$ from the far side to the near side from $\tilde o$.
In those cases we have the good positional relation between $\widetilde T(p,q)$ and $T(\tilde p, \tilde q)$.
In \S 7 we give the proof of our theorems. 
Our assumption (2.1) ensures that $\widetilde T(p,q)$ crosses $Cut (\tilde p)$ from the far side to the near side from $\tilde o$. The assumption is used only in the proofs of Assertions \ref{Ass2}.
In \S 8 we show some corollaries concerning the maximal perimeter and diameter as applications of our theorems. Those are the Riemannian version of Corollary \ref{App-1}, and we give the proof of  Corollary \ref{App1}.

Basic tools in Riemannian Geometry are referred to [1].
 
\section{Notations and statements} 
Let $(\widetilde M, \tilde o)$ be a surface of revolution with a geodesic polar 
coordinate system $(r, \theta)$ around $\tilde o$. Its metric is given 
by 
\[ 
ds^2 = dr^2 + f(r)^2d\theta^2, 
\] 
where $f(r) > 0$, $0 < r < \ell \le \infty$, $\theta \in S^1$ and 
$f : [0, \ell ) \longrightarrow \mathbb{R}$ satisfies the Jacobi equation 
\[ 
f''+Kf=0, \quad\quad f(0)=0, \quad f'(0)=1. 
\] 
In addition, 
\[ 
f(\ell)=0, \quad f'(\ell)=-1 \quad\quad \mbox{if} \quad\quad \ell < \infty. 
\] 
The function $K$ is called the {\it radial curvature function} of $\widetilde 
M$. 
 
Let $(M,o)$ be a complete Riemannian manifold with a base point at $o$. A {\it 
radial plane} $\Pi \subset T_pM$ at a point $p \in M$ is by definition a plane 
containing a vector tangent to a minimizing geodesic segment emanating from $o$ 
where $T_pM$ is the tangent space of $M$ at $p$. A {\it radial sectional 
curvature} $K_M(\Pi)$ is by definition a sectional curvature with respect to a 
radial plane $\Pi$. We say that $(M,o)$ is {\it referred} to $(\widetilde M, 
\tilde o)$ if every radial sectional curvature at $p \in M$ is bounded below by 
$K(d(o,p))$, namely, $K_M(\Pi) \ge K(d(o,p))$ where $d(o,p)$ is by definition 
the distance between $o$ and $p$. 

A triple of minimizing geodesic 
segments $T(o,p)\cup T(o,q) \cup T(p,q)$ joining points $o, p, q \in M$ is called a {\it geodesic triangle} and 
denoted by $\triangle opq$. A geodesic triangle $\triangle \tilde o \tilde p 
\tilde q \subset \widetilde M$ is called a {\it comparison triangle 
corresponding to} $\triangle opq \subset M$ if the corresponding edges have the 
same lengths, namely, 
\[ 
d(o,p)=d(\tilde o,\tilde p), \quad\quad 
d(o,q)=d(\tilde o,\tilde q), \quad\quad 
d(p,q)=d(\tilde p,\tilde q). 
\]
In the reference surface of revolution $\widetilde M$ every geodesic triangle $\triangle \tilde o \tilde p \tilde q$ bounds the region because $\widetilde M$ is simply connected and the dimension of $\widetilde M$ is two. The region is called a {\em triangle domain} and denoted by the same symbol $\triangle \tilde o \tilde p \tilde q$.

With respect to a point $\tilde p \in \widetilde M$, we divide $\widetilde M$ into two parts as follows: 
\[ 
\widetilde M_{\tilde p}^+=[ \theta( \tilde p ) \le \theta \le \theta( \tilde p 
)+ \pi ], \quad\quad 
\widetilde M_{\tilde p}^-=[ \theta( \tilde p ) - \pi \le \theta \le \theta( 
\tilde p ) ]. 
\] 
Here we set 
\[ 
[ a \le \theta \le b ]=\{ \tilde x \in \widetilde M \; | \; a \le \theta( \tilde 
x ) \le b \}. 
\] 
The pair of distance functions $\tilde x \mapsto (d(\tilde o, \tilde x), 
d(\tilde p, \tilde x))$, $\tilde x \in \widetilde M$, defines a Lipschitz chart on 
the interiors ${\rm Int}(\widetilde M_{\tilde p}^{\pm})$ of $\widetilde 
M_{\tilde p}^{\pm}$ respectively (see Lemma \ref{circle} (1)). For an arbitrary fixed point $p \in M$ and 
$\tilde p \in \widetilde M$ we define the maps : 
\[ 
F_p : M \rightarrow \mathbb{R}^2, \quad\quad  F_p(x)=(d(o,x), d(p,x)), \quad 
x \in M, 
\] 
and 
\[ 
\widetilde F_{\tilde p} : \widetilde M_{\tilde p}^+ \rightarrow 
\mathbb{R}^2, \quad\quad \widetilde F_{\tilde p}(\tilde x)=(d(\tilde o, \tilde 
x), d(\tilde p, \tilde x)), \quad \tilde x \in \widetilde M_{\tilde p}^+, 
\] 
\[ 
\widetilde G_{\tilde p} : \widetilde M_{\tilde p}^- \rightarrow 
\mathbb{R}^2, \quad\quad \widetilde G_{\tilde p}(\tilde x)=(d(\tilde o, \tilde 
x), d(\tilde p, \tilde x)), \quad \tilde x \in \widetilde M_{\tilde p}^-. 
\] 
Then, $F_p$, $\widetilde F_{\tilde p}$ and $\widetilde G_{\tilde p}$ are 
Lipschitz continuous, both $\widetilde F_{\tilde p}$ and $\widetilde G_{\tilde 
p}$ are injective and their inverse maps are locally Lipschitz continuous. 
 
A unit speed minimizing geodesic segment from $p$ to $q$ is denoted by 
$T(p,q)(t)$, $0 \le t \le d(p,q)$, where $T(p,q)(0)=p$ and $T(p,q)(d(p,q))=q$. 
Also $T(p,q)$ is identified with its image $\{ T(p,q)(t) \, | \, 0 \le t \le 
d(p,q) \}$. 

For an arbitrary fixed point $p \in M$ we set $\tilde p=(d(o,p), 
0) \in \widetilde M$. If $T(p,q) \subset F_p{}^{-1}( \widetilde F_{\tilde p}(\widetilde 
M_{\tilde p}^+))$, we then define a curve $\widetilde T(p,q)$ in $\widetilde M$ 
such that 
\[ 
\widetilde T(p,q)(t)=\widetilde F_{\tilde p}{}^{-1}\circ F_p(T(p,q)(t)),\quad 0 \le t \le d(p,q). 
\] 
 
Obviously we have $\widetilde T(p,q)(0)=\tilde p$. The reference map $\Phi $ 
used in \S 1 is just $\widetilde F_{\tilde p}{}^{-1}\circ F_p$ which 
depends on the choice of $\tilde p \in \widetilde M$ corresponding to $p \in M$.
Namely, if $\tilde p_1 = \tau_{\theta}(\tilde p)$ for some rotation $\tau_{\theta}$ of $\widetilde M$ around $\tilde o$, then  $\widetilde F_{\tilde p_1}{}^{-1}\circ F_p(T(p,q)(t))  =\tau_{\theta}(\widetilde F_{\tilde p}{}^{-1}\circ F_p(T(p,q)(t) )$).

It is convenient to use the expression $\widetilde F_{\tilde p}{}^{-1}\circ F_p$ 
for defining the reference reverse curve. 
Setting $\tilde q =\widetilde T(p,q)(d(p,q))$, we have the {\it reference 
reverse curve} $\widetilde R(p,q)$ of $T(p,q)$ which is given by 
\[ 
\widetilde R(p,q)(t)=\widetilde G_{\tilde q}{}^{-1}\circ F_q(T(p,q)(d(p,q)-t)),\quad 0 \le t \le d(p,q).
\] 
 We then have $\widetilde R(p,q)(0)=\tilde q$, 
$\widetilde R(p,q)(d(p,q))=\tilde p$. Both $\widetilde T(p,q)$ and $\widetilde 
R(p,q)$ are curves connecting $\tilde p$ and $\tilde q$ in $\widetilde M_{\tilde 
p}^+$. 
Notice that $\widetilde T(p,q) \not= \widetilde R(p,q)$, in general, as point 
sets in $\widetilde M$. 
 
A set $C$ is said to be {\em parameterized by the angle coordinate} $\theta $ if 
$C \cap [\theta = a]$ contains at most one point where $[\theta = a]=\{ \tilde x 
\in \widetilde M \, | \, \theta (\tilde x )=a \}$. Let two sets $C_1$ and $C_2$ be
parameterized by the angle coordinate. We then define the positional relation between $C_1$ and $C_2$ by $C_1 \le C_2$ if 
 $r(C_1 \cap [\theta = a]) \le r(C_2 \cap [\theta = a])$ for all 
$a \in \mathbb{R}$ with $C_1 \cap [\theta = a]\not= \emptyset $ and $C_2 \cap 
[\theta = a]\not= \emptyset $. \par
 
We say that a point $\tilde q$ in $\widetilde M$ is a {\it cut point} of $\tilde 
p$ if any extension of a minimizing geodesic segment $T(\tilde p, \tilde q)$ is not minimizing. Let 
$Cut(\tilde p)$ denote the set of all cut points of $\tilde p \in \widetilde M$. It is well known that $Cut(\tilde p)$ carries the structure of a tree in $\widetilde M$. 
All edges of $Cut(\tilde p)\cap {\rm Int}(\widetilde M^{\pm}_{\tilde p})$ and all non-meridian geodesics in $\widetilde M^+_{\tilde p}$ are parameterized by the angle coordinate.\par
 
For $r > d(o,p)$ we define an ellipsoid in $M$ by 
\[ 
E(o,p \, ;r)=\{ x \in M \, | \, d(o,x)+d(p,x)=r \} 
\] 
and the distance function to $o$ restricted to $E(o,p \, ; r)$ by 
$d_r(x)=d(o,x),\;x \in E(o,p \, ; r)$. Let $E_p(r)$ be the set of 
all points where $d_r$ attains local maximums and set 
$E(p)=\cup_{r>d(o,p)}E_p(r)$. We will have $E(p) \subset Cut (p)$ (see Lemma \ref{notmaximum}).

By using these notations we will prove the following theorem which is a 
restatement of Theorem \ref{main1}. \par

\begin{thm}\label{restate1}
Assume that a complete pointed Riemannian manifold $(M,o)$ is referred to a 
surface of revolution $(\widetilde M, \tilde o)$. Let $p \in M$. Suppose 
\begin{equation} 
F_p(E(p))\cap \widetilde F_{\tilde p}(Cut(\tilde p)\cap {\rm Int}(\widetilde 
M_{\tilde p}^+))=\emptyset. 
\end{equation}
Then, there exists a minimizing geodesic segment $T(\tilde p, \tilde q)$ in 
$\widetilde M_{\tilde p}^+$ such that 
\begin{equation} 
\widetilde T(p,q) \ge T(\tilde p, \tilde q) 
\quad\quad 
\mbox{and} 
\quad\quad 
\widetilde R(p,q) \ge T(\tilde p, \tilde q). 
\end{equation} 
holds for every minimizing geodesic segment $T(p,q)$, $q \in M$. 
\end{thm} 

\begin{rem}\label{ordinary}
{\rm 
The relation {\rm (2,2)} is nothing but the Alexandrov convexity property. 
Namely, we have from {\rm (2.2)} 
\begin{equation*} 
d(o,T(p,q)(t)) \ge d(\tilde o, T(\tilde p, \tilde q)(t)), \quad d(o,T(q,p)(t)) 
\ge d(\tilde o, T(\tilde q, \tilde p)(t)) 
\end{equation*} 
for all $t \in [0, d(p,q)]$ (see Lemma \ref{baseangle} (2)). Then the angle estimates at the corners $p$ and $q$ 
of $\triangle opq$ are obtained by the above relations (see Lemma \ref{baseangle} (3)). 
} 
\end{rem}

Moreover, the angle estimate at $o$ is obtained, also. The following theorem is 
the refined statement of Theorem \ref{main2}. 
 
\begin{thm}\label{restate2}
Under the same assumptions as in Theorem \ref{restate1}, every geodesic triangle $\triangle 
opq \subset M$ admits its comparison triangle $\triangle \tilde o \tilde p 
\tilde q \subset \widetilde M$ such that 
\begin{equation} 
\angle opq \ge \angle \tilde o \tilde p \tilde q, \quad\quad 
\angle oqp \ge \angle \tilde o \tilde q \tilde p, \quad\quad 
\angle poq \ge \angle \tilde p \tilde o \tilde q. 
\end{equation} 
\end{thm} 
 
We emphasize that {\rm (2.3)} is obtained only by the radial curvature with 
respect to $o$. 
 
\begin{rem}\label{isometric}
{\rm 
Under the same assumptions as in Theorem \ref{restate1}, if $\widetilde T(p,q)\cap 
T(\tilde p, \tilde q) \not= \{ \tilde p, \tilde q \}$ for a minimizing geodesic 
segment $T(p,q)$ in $M$, then $\widetilde T(p,q)=T(\tilde p, \tilde q)$ and a 
geodesic triangle $\triangle opq$ in $M$ bounds a totally geodesic 2-dimensional submanifold which is isometric to a comparison triangle domain
$\triangle \tilde o \tilde p \tilde q$ in 
$\widetilde M$ corresponding to $\triangle opq$ (see Lemma \ref{comparison}).
} 
\end{rem}

\begin{rem}\label{vonMangoldt}
{\rm 
If $\widetilde M$ is the standard 2-sphere, the flat plane or the Poincar\'e disk, then 
every point $\tilde o \in \widetilde M$ is viewed as a base point of $\widetilde 
M$ and any point $p \in M$ satisfies $Cut(\tilde p)\cap {\rm Int}(\widetilde M_{\tilde p}^+)=\emptyset$.
We say that a surface 
of revolution $\widetilde M$ is a {\it von-Mangoldt surface} if its radial 
curvature function is monotone non-increasing. 
Every point on a von-Mangoldt 
surface of revolution $\widetilde M$ satisfies $Cut(\tilde p)\cap {\rm Int}(\widetilde M_{\tilde p}^+)=\emptyset$ (see \cite{T}). Thus, Theorem \ref{restate2} implies that {\rm 
(2.3)} holds for every geodesic triangle $\triangle opq$. This result was first 
obtained in  \cite{IMS}. The angle estimate at the base point in a sector without cut 
points has been obtained by Kondo and Tanaka \cite{KT2}. 
} 
\end{rem} 
 
\begin{rem}\label{pole}
{\rm 
Assume that $o \in M$ is a pole of $M$. Namely, the exponential map ${\rm exp}_o 
: T_oM \rightarrow M$ at $o$ is a diffeomorphism. Then $E(p)$ for every 
point $p \not= o$ is the subray from $p$ of the meridian passing through $p$ 
(see Lemma \ref{notmaximum}). We then have (2.1) for every $p \in M$, $p \not= o$, and (2.3) 
for every $\triangle opq$. The same fact holds for a compact Riemannian manifold 
$M$ if the parameters of the first conjugate points to $o$ along any unit speed 
geodesics emanating from $o$ are constant. 
} 
\end{rem} 
 
Remark \ref{pole} suggests us to define a pointed Alexandrov space $(M,o)$ with radial 
curvature bounded below by a function $K$ as follows. Let $(\widetilde M, \tilde 
o)$ be a surface of revolution with radial curvature function $K$. We say that an
Alexandrov space $(M,o)$ with curvature locally bounded below is a {\it pointed Alexandrov space with radial 
curvature bounded below by the function} $K$ if the following condition is satisfied:
\begin{enumerate}
\item Every geodesic triangle $\triangle opq$ in $M$ admits its comparison triangle $\triangle \tilde o \tilde p \tilde q$ in $\widetilde M$ satisfying (2.2). 
\item Conversely, for every geodesic triangle $\triangle \tilde o \tilde p \tilde q$ whose vertices $\tilde p$ and $\tilde q$ in $\widetilde M$ are the reference points $p$ and $q$ in $M$, respectively, there exists a geodesic triangle $\triangle opq$ in $M$ satisfying (2.2).
\end{enumerate}   
This definition makes 
sense because of Remark \ref{pole}. In fact, if $K_1$ is any function less than or equal 
to $K$, then Remark \ref{pole} ensures that $(M,o)$ is a pointed Alexandrov space with 
radial curvature bounded below by the function $K_1$. 
Using this notion, we have Corollaries \ref{App-1} and \ref{App1} in \S 1. 

\section{Circles and ellipses}

Let $S(\tilde p \, ; a)=\{ \tilde x \in \widetilde M \, | \,  d(\tilde p, 
\tilde x) = a \}$ be the metric $a$-circle centered at $\tilde p$ and $S(\tilde p, a)^{\pm}=S(\tilde p \, ; a)\cap\widetilde M_{\tilde p}^{\pm}$. 
Like this, we often write $X^{\pm}:= X\cap \widetilde M^{\pm}_{\tilde p}$ for a set $X \subset \widetilde M$.
Let $I_{\tilde p, a}^{\pm}=\{ u \in  [r(\tilde p)-a, r(\tilde p)+a] \, | \, S(\tilde p \, ; a)^{\pm}\cap [r=u] \not= \emptyset \}$ and $S_{\tilde p, a}^{\pm}(u)=S(\tilde p \, ; a)^{\pm}\cap 
[r=u]$ for $u \in I_{\tilde p, a}^{\pm}$. 
We will show that $S_{\tilde p, a}^{\pm} : I^{\pm}_{\tilde p, a} \rightarrow \widetilde M^{\pm}_{\tilde p}$ is a union of curves. Obviously, $I_{\tilde p, a}^+=I_{\tilde p,a}^- =:I_{\tilde p,a}$. In general, $S(\tilde p \, ; a)$ is not necessarily connected, and then $I_{\tilde p, a}$ is the union of some intervals and points contained in $[r(\tilde p)-a, r(\tilde p)+a]$ as seen in Lemma \ref{circle} below. When $\ell < \infty $, we set $\tilde o_1$ to be the antipodal vertex of $\tilde o$, namely $r(\tilde o_1)=\ell$. Obviously, $S(\tilde o \, ; a)=S(\tilde o_1 \, ; \ell - a)=[r = a]$, $0 \le a \le \ell$.
 
\begin{lem}\label{circle}
Let $\tilde p \in \widetilde M$, $\tilde p \not= \tilde o, \, \tilde o_1$. The metric circles 
in $\widetilde M$ satisfy the following properties. 
\begin{enumerate} 
\item[(1)]  If an $r_1$-parallel circle $c=[r=r_1] \subset \widetilde 
M^+_{\tilde p}$ is parameterized as $c(\varphi )=[r=r_1]\cap[\theta = \varphi ]$ 
for any $ \varphi$, then $d(\tilde p, c(\varphi ))$ is strictly increasing in 
$\varphi \in [\theta (\tilde p), \theta (\tilde p)+\pi]$. 
\item[(2)] Each of $S_{\tilde p, a}^+(u)$ and $S_{\tilde p, a}^-(u)$ consists of only one point for any $u \in I_{\tilde p, a}$. In particular, 
each of $S_{\tilde p, a}^+$ and $S_{\tilde p, a}^-$ is the union of curves in 
$\widetilde M$ with parameter $u \in I_{\tilde p, a}$.
\item[(3)] Let $\tilde q \in \widetilde M_{\tilde p}^+$ and let $s > 0$ and $t > 0$ satisfy $s +t=d(\tilde p, \tilde q)$. If $\tilde z \in S(\tilde p, s)^+ 
\cap S(\tilde q, t)^-$, then there exists the unique minimizing geodesic segment 
$T(\tilde p, \tilde q)$ passing through $\tilde z$. 
\end{enumerate} 
\end{lem} 

\begin{proof} 
Since all meridians $[\theta = \varphi ]$ are geodesics which intersect the parallel circles $[r = r_1]$ orthogonally and hence the minimizing geodesic segments from $\tilde p$ intersect the parallel circle $c=[r = r_1]$ with the angles less than $\pi /2$, we have (1) from the first variation formula. 

This fact (1) implies that each of 
$S_{\tilde p, a}^+(u)$ and $S_{\tilde p, a}^-(u)$ consists of at most one point 
for any $u \in [r(\tilde p)-a, r(\tilde p)+a]$. This proves (2). 

Then (3) follows from the fact 
\[ 
d(\tilde p, \tilde z) + d(\tilde z, \tilde q) = s + t = d(\tilde p, \tilde q). 
\]
This completes the proof.
\end{proof} 
 
Let $T(p, q)^\cdot (0)$ denote the tangent vector of the curve $T(p,q)(t)$ at $t=0$.
We define a map $g : [0, \pi] \rightarrow S(\widetilde p \, ; a)^+\cup \{ \phi \}$ as follows: If $\omega \in [0, \pi]$ is the angle of $T(\tilde p, \tilde o)\,\dot{}(0)$ with $T(\tilde p, \tilde q)\,\dot{}(0)$ for some point $\tilde q \in S(\widetilde p \, ; a)$, then $g(\omega)=\tilde q$. If $\omega $ satisfies $\omega_1 \le \omega \le \omega_2$ for some $\omega_1$ and $\omega_2$ with $g(\omega_1)=g(\omega_2)$, then $g(\omega )=g(\omega_1)$. 
Otherwise, $g(\omega )=\phi$ where $\phi$ is the dummy. The connected components of $[0,\pi]\smallsetminus g^{-1}(\phi)$ corresponds to those of $I_{\tilde p,a}$.
If $\bar g : [0, \pi] \smallsetminus g^{-1}(\phi ) \rightarrow I_{\tilde p, a}$ is the map given by $g(\omega )=S^+_{\tilde p, a}(\bar g(\omega ))$, then $\bar g$ is monotone nondecreasing in each connected component.

Let $B(\tilde o, \tilde p \, ;  a)\subset\widetilde M$ for $a>d(\tilde o, \tilde p)$ be the domain given by $B(\tilde o, \tilde p \, ; a)=\{ \tilde q \, | \, d(\tilde  o,\tilde q)+d(\tilde p,\tilde q) \le a \}$. 
When $\ell < \infty $, the function $d(\tilde p, \tilde q) + d(\tilde o, \tilde q)$, $\tilde q \in \widetilde M$, attains the maximum $2\ell - d(\tilde p, \tilde o)$ at $\tilde o_1$. 
We therefore have $B(\tilde o, \tilde p \, ; a) \supset \widetilde M$ for every $a \ge 2\ell - d(\tilde p, \tilde o)$. If $a < 2\ell - d(\tilde p, \tilde o)$, then $\tilde o_1 \not\in B(\tilde o, \tilde p \, ; a)$. 

\begin{lem}\label{ellipse}
Let $\tilde p \not= \tilde o, \, \tilde o_1$.
The ellipses $E(\tilde o, \tilde p \, ; a)$, $d(\tilde o, \tilde p) < a < 2\ell - d(\tilde o, \tilde p)$, in $\widetilde M$ has the following properties. 
\begin{enumerate} 
\item[(1)]  $B(\tilde o, \tilde p \, ; a)$ is star-shaped around $\tilde p$ and $\tilde o$. Namely, $T(\tilde p, \tilde q) \subset B(\tilde o, \tilde p \, ; a)$ and $T(\tilde o, \tilde q) \subset B(\tilde o, \tilde p \, ; a)$ for any $\tilde q \in B(\tilde o, \tilde p \, ; a)$. 
If $\tilde q \in E(\tilde o, \tilde p \, ; a)$, then $T(\tilde p, \tilde q )\cap E(\tilde o, \tilde p \, ; a)=\{ \tilde q \}$ and $T(\tilde o, \tilde q )\cap E(\tilde o, \tilde p \, ; a)=\{ \tilde q \}$. Furthermore, $\angle \tilde p \tilde q \tilde o \not= \pi$.
\item[(2)] The intersection $E(\tilde o, \tilde p \, ; a)\cap [\theta=\varphi]$ is a single point for all $\varphi \in [\theta (\tilde p)-\pi , \theta (\tilde p)+\pi ]$. If $(r(\varphi), \varphi)=E(\tilde o, \tilde p \, ; a)\cap [\theta=\varphi]$, 
then $r(\theta (\tilde p)-\varphi)=r(\theta (\tilde p)+\varphi)$ for $\varphi \in [0, \pi]$. Moreover, $r(\varphi )$ is monotone increasing for $\varphi \in [\theta (\tilde p)-\pi , \theta (\tilde p)]$ and monotone decreasing for $\varphi \in [\theta (\tilde p) , \theta (\tilde p)+\pi]$. In particular, $r(\theta (\tilde p))=(a+d(\tilde o, \tilde p))/2$ is the maximum and $r(\theta (\tilde p)\pm \pi)$ is the minimum.
\item[(3)] If $e^{\pm}(u)=E(\tilde o, \tilde p \, ; a)\cap \widetilde M_{\tilde p}^{\pm}\cap [r=u]$ for $u \in [r(\theta (\tilde p)+\pi) , r(\theta (\tilde p))]$, then the function $d(\tilde p, e^{\pm}(u))$ is monotone decreasing in $u \in [r(\theta (\tilde p)+\pi) , r(\theta (\tilde p))]$.
\item[(4)] Set  $\tilde q_1= (r(\theta (\tilde p)\pm\pi) , \theta(\tilde p)\pm \pi )$. Let $b \in (a-r(\theta (\tilde p)), d(\tilde p, \tilde q_1))$. 
Then, $S(\tilde p, b)$ crosses $E(\tilde o, \tilde p \, ; a)$ once in each of $\widetilde M_{\tilde p}^+$ and $\widetilde M_{\tilde p}^-$. 
If $b=(a-d(\tilde o, \tilde p))/2$, then $S(\tilde p, b) \subset B(\tilde o, \tilde p \, ; a)$ and  $S(\tilde p, b) \cap E(\tilde o, \tilde p \, ; a )=\{ (r(\theta (\tilde p)), \theta (\tilde p)) \}$. 
If $b=d(\tilde o, \tilde q_1)$, then $S(\tilde o, b) \subset B(\tilde o, \tilde p \, ; a)$ and $S(\tilde o, b) \cap E(\tilde o, \tilde p \, ; a)=\{ \tilde q_1 \}$.
\end{enumerate}
\end{lem} 

To be seen in Example \ref{Cylinder}, the third statement of (1) is not true, in general, for a complete pointed Riemannian manifold $(M, o)$.
Namely, $T(p, q)\cup T(q,o)$ may be a geodesic segment from $p$ to $o$ via $q$ with $q \in E(o, p \, ; a)$. 
 
\begin{proof} 
Let $\tilde q \in B(\tilde o, \tilde p \, ; a)$ and let $\tilde q' \in T(\tilde p, \tilde q)$. We then have 
\begin{eqnarray*} 
d(\tilde o, \tilde q') + d(\tilde p, \tilde q') &=& d(\tilde o, \tilde q') + 
d(\tilde p, \tilde q) - d(\tilde q, \tilde q') \\ 
& \le & d(\tilde o, \tilde q) + d(\tilde p, \tilde q) \le a. 
\end{eqnarray*} 
This means that $\tilde q' \in B(\tilde o, \tilde p \, ; a)$, and, hence, $T(\tilde p, \tilde q) \subset 
B(\tilde o, \tilde p \, ; a)$. In the same way we have $T(\tilde o, \tilde q) \subset B(\tilde o, \tilde p \, ; a)$. These prove the first part of (1). 

Suppose there exists a point $\tilde q' \in E(o, p \, ; a) \cap T(\tilde p, \tilde q) \smallsetminus \{ \tilde q \}$. 
Then the above inequality shows that $d(\tilde o, \tilde q')=d(\tilde o, \tilde q)+d(\tilde q, \tilde q')$.
From this we have $T(\tilde p, \tilde q)\cap T(\tilde o, \tilde q)\supset T(\tilde q', \tilde q)$, meaning that $T(\tilde p, \tilde q)\cup T(\tilde q, \tilde o)$ is a geodesic connecting $\tilde p$ and $\tilde o$ which is different from the meridian passing through $\tilde p$, a contradiction.
This proves the second part of (1).

The third part of (1) is obvious. In fact, if $\angle \tilde p \tilde q \tilde o = \pi$, then $d(\tilde p, \tilde q) + d(\tilde q, \tilde o)+ d(\tilde o, \tilde p)=2\ell$, meaning that $a=2\ell - d(\tilde o, \tilde p)$, a contradiction.

Notice that the function $f(r)=d(\tilde p, (r,\varphi))+r$ is monotone increasing in $r \in (0, \ell)$ with $f(r) > r(\tilde p)$ because of the first variation formula. Since $\sup \{ f(r) \, | \, r \in (0, \ell) \}=2\ell - d(\tilde o, \tilde p)$, we have the first part of (2). 

Let $\varphi_1$ and $\varphi_2$ be such that $\theta (\tilde p) \le \varphi_1 < \varphi_2 \le \theta (\tilde p)+\pi$. If $r(\varphi_1)=r(\varphi_2)$, then $d(\tilde p, (r(\varphi_1),\varphi_1))=d(\tilde p, (r(\varphi_2),\varphi_2))$, which contradicts Lemma \ref{circle} (1).
By the triangle inequality,
\begin{eqnarray*}
r(\varphi) &\le& d(\tilde o, \tilde p) + d(\tilde p, (r(\varphi), \varphi )) \\
&=& d(\tilde o, \tilde p) + a - r(\varphi),
\end{eqnarray*}
and, hence, we have $r(\varphi) \le (a + d(\tilde o, \tilde p))/2$ where the equality holds if and only if $\varphi = \theta (\tilde p)$.
These imply the other parts of (2) and (3).

If $b \in (a-r(\theta (\tilde p)), a-r(\theta (\tilde p)+\pi))$, then each curve $(r(\varphi), \varphi)$ for $\varphi \in [\theta (\tilde p)-\pi, \theta (\tilde p)]$ and $[\theta (\tilde p), \theta (\tilde p)+\pi]$ moves from the outside of $S(\tilde p, b)$ to its inside and from its inside to its outside, respectively. The property (2) of this lemma implies that the crossing point is unique in each curve, which proves the first part of (4).

In order to prove the second part of (4), let $b= (a-d(\tilde o, \tilde p))/2$ and $\tilde q \in S(\tilde p, b)$. We then have
\begin{eqnarray*}
d(\tilde o, \tilde q)+d(\tilde p, \tilde q) \le d(\tilde o, \tilde p) + 2d(\tilde p, \tilde q) =a
\end{eqnarray*}
and equality holding if and only if $d(\tilde o, \tilde p)+d(\tilde p, \tilde q) = d(\tilde o, \tilde q)$. 
This means that $S(\tilde p, b) \subset B(\tilde o, \tilde p \, ; a)$ and  $S(\tilde p, b) \cap E(\tilde o, \tilde p \, ; a )=\{ (r(\theta (\tilde p)), \theta (\tilde p)) \}$.

Let $b=d(\tilde o, \tilde q_1)$ and $\tilde q \in S(\tilde o, b) $, namely $d(\tilde o, \tilde q)=d(\tilde o, \tilde q_1)$. We then have, from Lemma \ref{circle} (1),
\begin{eqnarray*}
d(\tilde o, \tilde q)+d(\tilde p, \tilde q) \le d(\tilde o, \tilde q_1)+d(\tilde p, \tilde q_1) =a,
\end{eqnarray*}
and equality holding if and only if $\tilde q=\tilde q_1$. This proves the third part of (4). 
\end{proof} 

Let $\tilde p \not= \tilde o , \, \tilde o_1$ and $d(\tilde o, \tilde p) < a < 2\ell - d(\tilde o, \tilde p)$. 
The reference curves will be made in $\widetilde M_{\tilde p}^+$, so we work in $\widetilde M_{\tilde p}^+$.
From Lemma \ref{ellipse} the set $\Omega (\tilde o, \tilde p \, ; a)^{+}$ of all minimizing geodesic segments $T(\tilde p, \tilde q)$ from $\tilde p$ to points $\tilde q \in E(\tilde o, \tilde p \, ; a)^+$ is a totally ordered set with respect to the binary relation $\le $ in the set of curves in $\widetilde M^{+}_{\tilde p}$. 
The minimizing geodesic segments $T(\tilde p, \tilde q)$, $\tilde q \in E(\tilde o, \tilde p \, ; a)^+$, divide $B(\tilde o, \tilde p \, ; a)^+$ into two domains. 
Let $U(\tilde p, \tilde q)$ denote the greatest minimizing geodesic segment connecting $\tilde p$ and $\tilde q$ and $L(\tilde p,\tilde q)$ the least one. Namely $L(\tilde p, \tilde q) \le T(\tilde p, \tilde q) \le U(\tilde p, \tilde q)$ for every minimizing geodesic segment $T(\tilde p, \tilde q)$. If $\tilde q \not\in Cut(\tilde p)$, then $U(\tilde p, \tilde q)=L(\tilde p, \tilde q)$. 
If $U(\tilde p, \tilde q)\not=L(\tilde p, \tilde q)$ for a point $\tilde q \in Cut(\tilde p)^+$, then $B(\tilde o, \tilde p \, ; a)^+$ is divided into three domains $B_1$, $B_0$ and $B_2$. 
Here $B_1$ is the domain bounded by the meridian $[\theta = \theta (\tilde p)]$, $E(\tilde o, \tilde p \, ; a)^+$ and $U(\tilde p, \tilde q)$, $B_0$ is the biangle domain bounded by $U(\tilde p, \tilde q)$ and $L(\tilde p, \tilde q)$, $B_2$ is the domain bounded by the meridians $[\theta =\theta (\tilde p)]\cup [\theta = \theta (\tilde p)+ \pi ]$, $E(\tilde o, \tilde p \, ; a)^+$ and $L(\tilde p, \tilde q)$.

\begin{lem}\label{SegEllipse}
Let $\tilde p \not= \tilde o , \, \tilde o_1$ and $d(\tilde o, \tilde p) < a < 2\ell - d(\tilde o, \tilde p)$. Let $\tilde q \in E(\tilde o, \tilde p \, ; a)^+$. 
Let $\tilde q'$ be a sequence of points in $E(\tilde o, \tilde p \, ; a)^+$ such that $r(\tilde q') > r(\tilde q)$ $(resp., r(\tilde q') < r(\tilde q))$ and it converges to $\tilde q$.
Then the sequence of segments $T(\tilde p, \tilde q')$ converges to $U(\tilde p, \tilde q)$ $(resp., L(\tilde p, \tilde q))$.
\end{lem}

\begin{proof}
A subsequence of the sequence $T(\tilde p, \tilde q')$ converges to a minimizing geodesic segment $T(\tilde p, \tilde q)$. 
Since $B_1$ and $B_2$ are star-shaped around $\tilde p$, it follows that $T(\tilde p, \tilde q')$ is contained in either $B_1$ or $B_2$, depending on $r(\tilde q') > r(\tilde q)$ or $r(\tilde q') < r(\tilde q)$. From the definition of $U(\tilde p, \tilde q)$ and $L(\tilde p, \tilde q)$, it follows that $T(\tilde p, \tilde q)$ is one of $U(\tilde p, \tilde q)$ and $L(\tilde p, \tilde q)$. This shows that the sequence of minimizing geodesic segments  $T(\tilde p, \tilde q')$ converges to either $U(\tilde p, \tilde q)$ or $L(\tilde p, \tilde q)$.
\end{proof}

We now discuss the property of ellipses in $M$, which includes new ideas and play an important role.
The following lemma gives a sufficient condition for $q \not\in E_p(r)$, namely $q$ is not a local maximum point of the distance function $d_r$ to $o$ restricted to $E(o,p \, ;r)$. 

\begin{lem}\label{notmaximum}
Let $q \in E(o,p \, ; r) \subset M$. If there exists a point $u$ such that 
$d(p,u)+d(o,u) > r$ and $d(p,u)-d(p,q) < d(o,u)-d(o,q)$, 
then there exists a point $q' \in E(o,p \, ;r)$ such that $d(o,q') > d(o,q)$. In 
particular, if $q \not\in Cut(o)$ and $p \not\in T(o,q)$, then $q$ is not a local 
maximum point of $d_r$ on $E(o,p \, ;r)$. 
\end{lem}
 
We observe that the assumption $d(p,u)+d(o,q) < d(p,q)+d(o,u)$ means that $d(o, T(q,u)(t))$ increase further than $d(p, T(q,u)(t))$ for  $t \in [0,d(q,u)]$.
\begin{proof} 
We first prove that the set $E(o,p \, ; r)\cap T(p,u)$ consists of a single 
point, say $q'$. Suppose there exists a point $q'' \in E(o,p \, ; r)\cap T(p,u)$ with 
$q''\not=q'$. Assume without loss of generality that $p, q', q'', u$ are in this 
order in $T(p,u)$. Since 
\begin{eqnarray*} 
d(p,q')+d(o,q') &=& r \\ 
&=& d(p,q'')+d(o,q'') \\ 
&=& d(p,q')+d(q',q'')+d(o,q''), 
\end{eqnarray*} 
we have $d(o,q')=d(q',q'')+d(o,q'')$. This means that the minimizing geodesic segment 
$T(q', q'')$ is contained in both segments $T(o,q')$ and $T(p,u)$. In 
particular, $u \in T(o,q')$. This is a contradiction, because 
\begin{eqnarray*} 
r < d(p,u)+d(o,u) &=& d(p,q'')+d(q'',u)+d(o,u) \\ 
&=& d(p,q'')+d(o,q'') = r. 
\end{eqnarray*}

We should note that $q' \not= q$. In fact, if $q'=q$, we then have 
\begin{eqnarray*} 
d(q',u) &=& d(p,u) -d(p,q') \\
&=& d(p, u) - d(p, q) \\ 
& < & d(o,u) - d(o,q') \le d(q',u), 
\end{eqnarray*} 
a contradiction. 
 
We next prove that $d(o,q') > d(o,q)$. If $d(p,q') < d(p,q)$, we have 
\begin{eqnarray*} 
d(o,q') &=& r-d(p,q') \\ 
&>& r-d(p,q) \\ 
&=& d(o,q). 
\end{eqnarray*} 
Thus we suppose $d(p,q') \ge d(p,q)$. We then have 
\begin{eqnarray*} 
d(o,q') &\ge& d(o,u) - d(u,q') \\ 
&=& d(o,u) - d(p,u)+d(p,q') \\ 
&\ge& d(o,u)-d(p,u)+d(p,q) \\ 
&>& d(o,u)- d(o,u)+d(o,q) \\ 
&=& d(o,q). 
\end{eqnarray*} 
This completes the proof of the first part of the lemma. 
 
Assume that $q \not\in Cut(o)$ and $p \not\in T(o,q)$. Let $u$ be a point such that $q \in T(o,u)$ with $u \not= q$. We have
\begin{eqnarray*} 
d(p,u)+d(o,u) &=& d(p,u)+d(u,q)+d(q,o) \\ 
&\ge& d(p,q) + d(o,q) = r,
\end{eqnarray*}
where the equality holds if and only if $d(p,q)=d(p,u)+d(u,q)$.
By the triangle inequality, we have
\begin{eqnarray*} 
d(p,u)-d(p,q) \le d(q,u) = d(o,u)-d(o,q), 
\end{eqnarray*}
where the equality holds if and only if $d(p,u)=d(p,q)+d(q,u)$. \par

In order to prove that $q$ is not a local maximum point of $d_r$ on $E(o,p \, ;r)$, we have to discuss the equality cases. Suppose first that $d(p,u)+d(u,q)=d(p,q)$.
Then, $T(u,q) \subset T(p,q)\cap T(o,u)$, which means that $T(u,q) \subset E(o,p \, ; r)$. 
Namely, $T(p,q)\cup T(u,o)$ is a geodesic in $M$ connecting $p$ and $o$ which is not minimizing such that the subsegment from $u$ to $q$ is contained in $E(o, p \, ; r)$.
Such a geodesic will be seen in Example \ref{Cylinder}.
Since every point $q' \in T(q,u)\smallsetminus \{ q\}$ satisfies that $d(o,q') > 
d(o,q)$, the point $q$ is not a local maximal point of the function $d_r$. 

We next suppose $d(p,u)=d(p,q)+d(q,u)$. 
Then, $T(q,u) \subset T(p,u)\cap T(o,u)$, 
Since $p \not\in T(o,q)$, we have $o \in T(p,q) \subset T(p,u)$. 
Since $q \not\in Cut (p)$, the set 
$S(p,d(p,q))=\{ u' \, | \, d(p,u')=d(p,q) \}$ contains a set $U$ around $q$ which is homeomorphic to a disk with dimension $\dim M - 1$ and any point $u' \in U$ with 
$u' \not= q$ satisfies $d(o,u') > d(o,q)$. 
In fact, $d(o, u') > | d(p,u') - d(p, o) | = | d(p, q) - d(p, o) |=d(o, q)$.
Thus, we can find a point $u'$ near $q$ satisfying the assumption 
in the first part of the lemma, namely $d(p,u')+d(o,u') > r$ and $d(p,u')-d(p,q) < d(o,u')-d(o,q) $.
From these arguments we may assume without loss of generality that there exists a point $u'$ near $q$ satisfying the assumption in the first part.

It remains to find a point $q'' \in E(o,p \, ; r)$ near $q$ such that $d(o,q'') > d(o,q)$.
Let $u'$ be a sequence of points satisfying the assumption 
in the first part of the lemma and converging to $q$. Let $q'(u')=E(o,p \, ; r)\cap T(p,u')$ which satisfies $d(o,q'(u')) > d(o,q)$. 
The sequence of minimizing geodesic segments $T(q'(u'),u')$ converges to the point $q$ or it contains a subsequence converging to a minimizing geodesic segment $T(q',q)$ contained in $E(o,p \, ; r)$ as $u'$ goes to $q$. 
When the first case occurs, the existence of $q'(u')$ shows that $q$ is not a local maximum point of $d_r$. Suppose the second case happens.
If $q'' \in T(q',q)$, we then have 
\begin{eqnarray*}
d(o,q'') &=& r - d(p,q'') \\
&=& r - ( d(p,q) - d(q'',q)) \\
&>& r - d(p,q) = d(o,q).
\end{eqnarray*}
This implies 
that $q$ is not a local maximum point of $d_r$ on $E(o,p \, ; r)$. This completes the 
proof. 
\end{proof} 

The following example is helpful to understand what happens on ellipses as being larger. It should be noted that there exists a point $q \in E(o, p \, ; r)$ which cannot be an accumulation point of interior points of $B(o, p \, ; r)$.

\begin{exa}\label{Cylinder}
{\rm
We study how ellipses change in a flat cylinder as being larger.
Let $M=\{ (x,y,z) \in \mathbb{E}^3 \, | \, x^2+z^2=1 \}$. Let $o=(1, 0, 0)$ and $p=(0,2,-1)$. 
Then $Cut(o)=\{ (-1,y,0) \, | \, y \in \mathbb{R} \}$ and $Cut(p)=\{ (0,y,1) \, | \, y \in \mathbb{R} \}$. 
We identify $M$ with $\mathbb{E}^2/\Gamma $ where $\mathbb{E}^2 =\{ (x,y) \, | \, x, y \in \mathbb{R} \}$ and $\Gamma $ is the isometry group generated by a translation $\tau$ such that $\tau ((x,y))= (x,y+2\pi)$. 
The universal covering space $\pi : \mathbb{E}^2 \rightarrow M$ is given by $\pi ((y,\theta))= (\cos \theta ,y,\sin \theta)$.
The tangent plane $M_o$ is idetified with $\mathbb{E}^2$ also. 
Then $\widetilde C(o)=\{ (x,\pm\pi) \, | \, x \in \mathbb{R} \}$ is the tangent cut locus of $o$ and $U=\{ (x,y) \, | \, x \in \mathbb{R}, -\pi \le y \le \pi \}$ is the lift of the normal coordinate neighborhood of $o$, namely $\exp_o : U \smallsetminus \widetilde C(o) \longrightarrow M \smallsetminus Cut(o)$ is a diffeomorphism. 
If $\varphi = \exp_o|\,U$, then $\varphi^{-1}(o)=(0,0)=:o_0$ and $\varphi^{-1} (p)=(2,-\pi/2)=:p_0$ by this identification. 
Set $p_1=(2, 3\pi /2 )$. Further, $\varphi^{-1}(Cut(p))=\{ (x, \pi/2) \, | \, x \in \mathbb{R} \}$. Let $E(o,p \, ;r)=\{ w \, | \, F(w):=d(o,w)+d(p,w)=r \}$ and $B(o,p \, ;r)=\{ w  \, | \, F(w) \le r \}$ for each $r > d(o,p)$. 

Set $r_0=\min \{ F(w) \, | \, w \in Cut(o) \}$, $\{ a \}=E(o_0,p_0 \, ;r_0)\cap T(o_0,p_1)$, $\{ q_{1} \}=\widetilde C(o)\cap T(o_0,p_1)$. Let $\partial X$ denote the boundary of a subset $X$. 
Then $\varphi^{-1}(E(o,p \, ;r))$ changes for $r$ as follows.
\begin{enumerate}
\item $\varphi^{-1}(E(o,p \, ;r))=E(o_0,p_0;r)$ if $r$ satisfies $d(o,p) < r <r_0$.
\item $\varphi^{-1}(E(o,p \, ;r_0))=E(o_0,p_0 \, ;r_0)\cup T(a,q_{1})$.
\item $\varphi^{-1}(E(o,p \, ;r))=\partial (B(o_0,p_0;r)\cup B(o_0,p_1 \, ; r))\cap U$ if $r$ satisfies $r > r_0$.
\end{enumerate}
Let $q_{2}=\tau^{-1}(q_{1})$. If $q \in Cut(o)$ satisfies $F(q)=\min F \, | \, Cut(o)$, then $\varphi^{-1}(q)=\{ q_{1}, q_{2} \}$.
Moreover, $\varphi (T(o_0,q_{1})\cup T(q_{2},p_0))$ is a geodesic connecting $o$ and $p$ in $M$. The geodesic reflecting against $Cut(o)$ at $q$ in $M$ is identified with $\varphi (T(o_0,q_{2})\cup T(q_{2},p_0))$.
}
\end{exa}

It should be remarked that any sequence of points $q'_j$ such that $q'_j \in E(o,p \, ; r_j)$ for $r_j < r_0$ with $r_j \rightarrow r_0$ cannot converge to any point in $\varphi (T(a,q_{1})\smallsetminus \{ a, q_{1} \}) \subset E(o,p \, ; r_0)$.
Thus, we notice that there exists a geodesic triangle $\triangle opq$ with $q \in E(o, p \, ; r_0)$ such that it admits no sequence of geodesic triangles $\triangle opq_j$ with $q_j \in E(o, p \, ; r_j)$, $r_j < r_0$, converging to itself.

\section{Reference curves}
Let $(\widetilde M, \tilde o)$ be a surface of revolution with vertex $\tilde o$.
Throughout this section, we do not assume that $(M,o)$ is referred to $(\widetilde M, \tilde o)$ and $F_p(E(p))\cap 
\widetilde F_{\tilde p}(Cut(\tilde p)\cap {\rm Int}(\widetilde M_{\tilde 
p}^+))=\emptyset $. 
However, we assume that every minimizing geodesic segment $T(p,q)$ in consideration is contained in $F_p{}^{-1}( \widetilde F_{\tilde p}(\widetilde M_{\tilde p}^+))$. Therefore, $\widetilde T(p, q)(t)$ is 
defined for all $t \in [0, d(p,q)]$.
\begin{lem}\label{basic}
Let $\tilde q(t)=\widetilde T(p,q)(t)$, $0 \le t \le d(p,q)$. 
Let I denote the set of all parameters $t \in [0, d(p,q)]$ such that $\tilde q(t) \in {\rm Int}(\widetilde M^+_{\tilde p})$. The curves $\widetilde T(p,q)$ and $\widetilde R(p,q)$ satisfy the following 
properties. 
\begin{enumerate} 
\item[(1)] $I$ is an interval. $\theta (\tilde q(t))$ is monotone increasing for $t \in I$. More precisely, if $t_0:=\max \{ t \in [0, d(p,q)] \, | \, \theta (\tilde q(t))=\theta (\tilde p) \} > 0$, we then have two possibilities: 
If $\tilde q(t_0) \in T(\tilde p, \tilde o)$, then $T(p,q)$ is contained in the maximal minimizing geodesic segment $T_e(p,o)$ from $p$ through $o$ and $\widetilde T(p,q)$ is contained in the union of meridians $[\theta = \theta (\tilde p)]\cup [\theta = \theta (\tilde p)+\pi]$.
If $\tilde p \in T(\tilde o, \tilde q(t_0))$, then $T(p,q(t_0)) = T(o, q(t_0))\cap T(p,q)$ and $\widetilde T(p, q)([0,t_0])$ is contained in the meridian through $\tilde p$.
In addition, if $t'_0 :=\min \{ t \in [0, d(p,q)] \, | \, \theta (\tilde q(t))=\theta (\tilde p)+\pi \} < d(p,q)$, we then have the similar results as above by using $q$ and $\tilde q$ instead of $p$ and $\tilde p$.
\item[(2)] $\widetilde R(p,q)(t)$ is defined for all $t \in [0,d(p,q)]$.
\item[(3)] $d(\tilde p, \widetilde T(p,q)(t))=t$, $d(\tilde q, \widetilde 
R(p,q)(t))=t$,\; $0 \le t \le d(p,q)$. 
\item[(4)] $d(\tilde p, \widetilde T(p,q)(t))+d(\tilde q, \widetilde 
R(p,q)(d(p,q)-t))=d(p,q)=d(\tilde p, \tilde q)$. 
\item[(5)] $r(\widetilde T(p,q)(t))=r(\widetilde R(p,q)(d(p,q)-t))$.
\item[(6)] $\widetilde T(p,q) \supset \widetilde T(p,q')$ and $\widetilde R(p, q) \supset \widetilde R(q', q)$ for any point $q' \in T(p,q)$.
\end{enumerate} 
\end{lem} 
 
\begin{proof}
Let $q(t)=T(p,q)(t)$. We first prove that if there exist two parameters $t_1$ and $t_2$ such that $t_1 < t_2$ and $\theta (\tilde q(t_1))=\theta (\tilde q(t_2))$ or $\tilde q(t_1)=\tilde o$ (or $\tilde q(t_1)=\tilde o_1$ if $\ell < \infty$),  then $d(\tilde p, \tilde q(t_2)) = d(\tilde p, \tilde q(t_1)) + d(\tilde q(t_1),\tilde q(t_2))$, namely $\tilde q(t_1) \in T(\tilde p, \tilde q(t_2))$. In fact, since $\theta (\tilde q(t_1))=\theta (\tilde q(t_2))$ implies that $| r(\tilde q(t_2)) - r(\tilde q(t_1)) | = d(\tilde q(t_1),\tilde q(t_2))$, we have
\begin{eqnarray*}
d(\tilde p, \tilde q(t_2)) - d(\tilde p, \tilde q(t_1)) &=& d(p, q(t_2)) - d(p, q(t_1)) \\
&=& d(q(t_1),q(t_2)) \\
&\ge& | d(o,q(t_2))-d(o,q(t_1)) | \\
&=& | r(\tilde q(t_2)) - r(\tilde q(t_1)) | \\
&=& d(\tilde q(t_1),\tilde q(t_2)) \\
&\ge& d(\tilde p, \tilde q(t_2)) - d(\tilde p, \tilde q(t_1)),
\end{eqnarray*}
meaning that $\tilde q(t_1) \in T(\tilde p, \tilde q(t_2))$.
Thus, $T(\tilde p, \tilde q(t_2))$ is contained in the union of the meridians $[\theta = \theta (\tilde p)] \cup [\theta = \theta(\tilde p) + \pi]$. 
Therefore, $\theta (\tilde q(t))$ is monotone increasing in the interval $I \subset [0,d(p,q)]$ such that $\tilde q(I) \subset {\rm Int}(\widetilde M^+_{\tilde p})$.  

Suppose $t_0 >0$. We have to treat two cases; $\tilde q(t_0) \in T(\tilde p, \tilde o)$ and $\tilde p \in T(\tilde o, \tilde q(t_0))$. In the first case, we have
\begin{eqnarray*}
d(p, q(t_0))+d(q(t_0),o) &=& d(\tilde p, \tilde q(t_0))+d(\tilde q(t_0), \tilde o) \\
&=& d(\tilde p, \tilde o) \\
&=& d(p, o),
\end{eqnarray*}
meaning that $T(p,q(t_0)) \subset T(p, o)$. Therefore, $T(p,q) \subset T_e(p,o)$, and therefore $\widetilde T(\tilde p, \tilde q)$ is contained in the union of the meridian through $\tilde p$ and the meridian opposite to $\tilde p$.
In the second case, we have
\begin{eqnarray*}
d(o, p)+d(p,q(t_0)) &=& d(\tilde o, \tilde p)+d(\tilde p, \tilde q(t_0)) \\
&=& d(\tilde o, \tilde q(t_0)) \\
&=& d(o, q(t_0)),
\end{eqnarray*}
meaning that $T(p,q(t_0)) \subset T(o, q(t_0))$. Therefore, $\widetilde T(p, q)([0,t_0])$ is contained in the meridian through $\tilde p$.

It remains to prove that $\theta (\tilde q(t))=\theta (\tilde p) + \pi$ for $t > t'_0$ in case of $\theta (\tilde q(t'_0))=\theta (\tilde p) + \pi$. 
Suppose that there exists a parameter $t \in (t'_0, d(p,q))$ such that $\theta (\tilde p) < \theta (\tilde q(t)) < \theta (\tilde p)+\pi $. 
We then find a parameter $t_3 \in (0,t'_0)$ such that $\theta (\tilde q(t_3))=\theta (\tilde q(t))$ or $\tilde q(t_3)=\tilde o$ (or $\tilde q(t_3)=\tilde o_1$).
By the same argument as above, we have a contradiction. 
In particular, $\theta (q(t'_0))=\theta (\tilde q(d(p,q)))=\theta (\tilde p)+\pi$. From the argument above,  it follows that $T(\tilde p, \tilde q)$ is contained in the union of the meridians $[\theta = \theta (\tilde p)] \cup [\theta = \theta(\tilde p) + \pi]$.
We have proved (1).

In order to prove (2) we suppose that $\theta (\tilde p(t_0))=\theta (\tilde p)$ for some $t_0 \in [0, d(p,q))$ where $\tilde p(t)=\widetilde R(p,q)(t)$, $0 \le t \le d(p,q)$. Then, we have 
\begin{eqnarray*}
d(p,q) &=& t_0 + d(p, q) -t_0 \\
&=& d(\tilde q, \tilde p(t_0)) + d(\tilde p, \tilde q(d(p,q) - t_0)) \\
&\ge& d(\tilde q, \tilde p(t_0)) + d(\tilde p, \tilde p(t_0)) \\
&\ge& d(\tilde p, \tilde q) \\
&=& d(p,q),
\end{eqnarray*}
since $r(\tilde p(t_0))=r(\tilde q(d(p,q)-t_0))=d(o, q(d(p,q)-t_0))$ and $\tilde p(t_0)$ lies in the meridian through $\tilde p$ and because of Lemma \ref{circle} (1). As before, $\widetilde R([t_0, d(p,q)])$ lies on the meridian through $\tilde p$. This shows (2).

Since 
\begin{eqnarray*} 
&&(d(\tilde o, \widetilde T(p,q)(t)), d(\tilde p, \widetilde T(p,q)(t))) \\ 
&=& \widetilde F_{\tilde p}(\widetilde T(p,q)(t)) \\ 
&=& F_{p}(T(p,q)(t)) \\ 
&=& (d(o, T(p,q)(t)), d(p, T(p,q)(t))), 
\end{eqnarray*} 
we have 
\[ 
r(\widetilde T(p,q)(t))=d(o, T(p,q)(t)), \quad d(\tilde p, \widetilde 
T(p,q)(t))=t. 
\] 
Since 
\begin{eqnarray*} 
&&(d(\tilde o, \widetilde R(p,q)(t)), d(\tilde q, \widetilde R(p,q)(t))) \\ 
&=& \widetilde G_{\tilde q}(\widetilde R(p,q)(t)) \\ 
&=& F_{q}(T(p,q)(d(p,q)-t)) \\ 
&=& (d(o, T(p,q)(d(p,q)-t)), d(q, T(p,q)(d(p,q)-t))), 
\end{eqnarray*} 
we have 
\[ 
r(\widetilde R(p,q)(t))=d(o, T(p,q)(d(p,q)-t)), \quad d(\tilde q, \widetilde 
R(p,q)(t))=t. 
\] 
Thus we have (3) and 
\[ 
d(\tilde p, \widetilde T(p,q)(t))+d(\tilde q, \widetilde 
R(p,q)(d(p,q)-t))=d(p,q) 
\] 
which proves (4). Then (5) follows from 
\begin{eqnarray*} 
r(\widetilde T(p,q)(t)) &=& d(o, T(p,q)(t)) \\ 
&=& r(\widetilde R(p,q)(d(p,q)-t)). 
\end{eqnarray*}
Obviously, (6) follows from the definition of the reference curves and the reference reverse curves 
\end{proof} 

Let $q(t)=T(p,q)(t)$ and $\tilde q(t)=\widetilde T(p,q)(t)$,\; $0 \le t \le d(p,q)$.
Lemma \ref{basic} (3) shows that a geodesic triangle $\triangle \tilde o \tilde p \tilde 
q(t)$ in $\widetilde M$ is a comparison triangle corresponding to $\triangle 
opq(t)$ in $M$.
 
Let $\theta (t)=\theta ( \widetilde R(p,q)(d(p,q)-t))-\theta (\widetilde 
T(p,q)(t))$,\; $0 \le t \le d(p,q)$. The following lemma shows the difference between 
$\widetilde T(p,q)(t)$ and $\widetilde R(p,q)(t)$ in terms of $\theta (t)$. 
 
\begin{lem}\label{TandR}
The reference curves $\tilde q(t)=\widetilde T(p,q)(t)$ and $\widetilde R(p,q)(t)$ satisfy 
the following properties. 
\begin{enumerate} 
\item[(1)] $\theta (t) \ge 0$ for all $t \in [0, d(p,q)]$. Moreover, if $\theta (t) \not= 0$ at  $t \in (0,d(p,q))$, namely $\widetilde T(p,q)(t) \not= \widetilde R(p,q)(d(p,q)-t)$, then $T(\tilde p, \tilde q)$ does not cross the subarc of the parallel $[r=r(\tilde q(t))]$ in $\widetilde M^+_{\tilde p}$ joining $\widetilde T(p,q)(t)$ and $\widetilde R(p,q)(d(p,q)-t)$.
\item[(2)] $\theta (t)=0$ if and only if $\widetilde T(p,q)(t)=\widetilde 
R(p,q)(d(p,q)-t)$. Then the point is in a minimizing geodesic segment $T(\tilde 
p, \tilde q)$. 
\item[(3)] If there exists a point $\tilde q' \in \widetilde T(p,q) \cap T(\tilde p, \tilde q) \smallsetminus \{ \tilde p, \tilde q \}$ $($resp., $\widetilde R(p,q) \cap T(\tilde p, \tilde q) \smallsetminus \{ \tilde p, \tilde q \}$ $)$, then $\widetilde R(p,q)(d(p,q)-d(\tilde p, \tilde q'))=\tilde q'$ $($resp., $\widetilde T(p,q)(d(p,q)-d(\tilde p, \tilde q'))=\tilde q'$ $)$. 
\item[(4)] $\widetilde T(p,q) \ge T(\tilde p, \tilde q)$ if and only if 
$\widetilde R(p,q) \ge T(\tilde p, \tilde q)$. 
\item[(5)] $\widetilde T(p,q) \le T(\tilde p, \tilde q)$ if and only if 
$\widetilde R(p,q) \le T(\tilde p, \tilde q)$.
\end{enumerate} 
\end{lem} 
 
\begin{proof} 
It follows that $\theta (\tilde p) < \theta (\widetilde T(p,q)(t))$ and $\theta 
(\tilde p) < \theta (\widetilde R(p,q)(d(p,q)-t))$,\; $0 < t < d(p,q)$. 
Suppose the first part of (1) is false. Then Lemma \ref{basic} (5) and Lemma \ref{circle} (1) show that 
\[ 
d(\tilde p, \widetilde T(p,q)(t)) > d(\tilde p, \widetilde R(p,q)(d(p,q)-t)) 
\] 
for some $t$. This contradicts Lemma \ref{basic} (4), since 
\begin{eqnarray*} 
d(\tilde p,\tilde q)&=&d(\tilde p, \widetilde T(p,q)(t))+d(\tilde q, \widetilde 
R(p,q)(d(p,q)-t) \\ 
&>& d(\tilde p, \widetilde R(p,q)(d(p,q)-t))+d(\tilde q, \widetilde 
R(p,q)(d(p,q)-t) \\ 
&\ge& d(\tilde p, \tilde q). 
\end{eqnarray*}

We prove the second part of (1). Suppose that there exists a point $T(\tilde p, \tilde q)(t_0)$ lying on the parallel circle joining $\tilde q(t)$ and $\widetilde R(p,q)(d(p,q)-t)$ for some $t \in [0, d(p,q)]$. 
We then have $r(T(\tilde p, \tilde q)(t_0))=r(\tilde q(t))$. 
Since $\theta (\widetilde T(p, q)(t))$ is monotone increasing in $t \in [0, d(p,q)]$, we have $\theta (\tilde p) \le \theta (\tilde q(t))$ and $\theta (\widetilde R(p,q)(d(p,q)-t)) \le \theta (\tilde q)$. 
Since $\theta (\tilde q(t)) < \theta (T(\tilde p, \tilde q)(t_0) < \theta (\widetilde R(p,q)(d(p,q)-t))$, it follows from Lemma \ref{circle} (1) that $t < t_0$ and $d(\tilde p, \tilde q) - t_0 > d(p, q)-t$. 
Hence, we have $d(\tilde p, \tilde q) > d(p,q)$, a contradiction.

If $\theta (t)=0$, then the equality holds in the above inequalities, and hence 
 $\widetilde T(p,q)(t)=\widetilde R(p,q)(d(p,q)-t)$. The converse is 
trivial. The second part of (2) follows from Lemma \ref{basic} (4).

We prove (3). Let $q' \in T(p,q)$ correspond to $\tilde q'$, namely, $d(p,q')=d(\tilde p, \tilde q')$. Recall that $\tilde p':=\widetilde R(p,q)(d(p,q)-d(p, q'))$ is the point in $\widetilde M^-_{\tilde q}$ such that $d(\tilde o, \tilde p')=d(o, q')$ and $d(\tilde q, \tilde p')=d(q, q')$. Since $d(q,q')=d(p,q)-d(p,q')=d(\tilde p, \tilde q)-d(\tilde p, \tilde q')=d(\tilde q, \tilde q')$ and $d(\tilde o, \tilde q')=d(o, q')$, we have $\tilde p'=\tilde q'$. In the same way we can prove the other case.
 
For the proof of (4) and (5), we suppose for indirect proof that $\widetilde 
T(p,q) \ge T(\tilde p, \tilde q)$ and $\widetilde R(p,q)(s) < T(\tilde p, \tilde 
q)$ for some $s \in (0, d(p,q))$. Then, there exists a point $\tilde z \in 
T(\tilde p, \tilde q)$ such that $\theta (\widetilde T(p,q)(d(p,q)-s)) \le 
\theta (\tilde z) < \theta (\widetilde R(p,q)(s))$, contradicting the second part of (1).
The remainder cases are proved in the same way. 
\end{proof} 
 
%%%%%%%%%%%%%%%%%%%%%%%%%%%%%%% 
\begin{lem}\label{baseangle}
Let $q(t)=T(p,q)(t)$ be a minimizing geodesic segment in 
$M$. Assume that $\widetilde T(p, q(t)) \ge T(\tilde p, \tilde q(t))$ for any $t \in [0, d(p,q)]$.
Set $\tilde p(t)=\widetilde R(p,q)(t)$.
Then, we have 
\begin{enumerate} 
\item[(1)] $T(\tilde p, \tilde q(t)) \ge T(\tilde p, \tilde q(s))$ and $T(\tilde p(t), \tilde q) \ge T(\tilde p(s), \tilde q)$ for any $t < s$. 
\item[(2)] $r(\tilde q(t)) \ge r(T(\tilde p, \tilde q)(t))$ and $r(\tilde p(t)) \ge r(T(\tilde p, \tilde q)(d(p,q)-t))$, $0 \le t \le d(p,q)$. 
\item[(3)] $\angle opq \ge \angle \tilde o \tilde p \tilde q$ and $\angle oqp \ge \angle \tilde o \tilde q \tilde p$. 
\end{enumerate} 
\end{lem} 
 
\begin{proof} 
We notice that $\widetilde T(p, q(t))\cup T(\tilde p, \tilde q(t))$ bounds a 
figure $\Omega$ in $\widetilde M$. We see that $T(\tilde p, \tilde q(s))$ cannot 
pass through any interior point of $\Omega$. In fact, if $T(\tilde p, \tilde q(s))$ contains an interior point in $\Omega$, then $T(\tilde p, \tilde q(s))$ meets $\widetilde T(p, q(t))$ at $\widetilde T(p, q(t))(t_0)$ for some $t_0 \in (0,t]$, because $T(\tilde p, \tilde q(t))\cap T(\tilde p, \tilde q(s)) = \{ \tilde p \}$ and $\tilde q(s) \not\in \Omega$. Since $\widetilde T(p, q(t))$ is a subarc of $\widetilde T(p, q(s))$, we have $\widetilde T(p, q(t)) \ge \widetilde T(p, q(s))$ and, hence, $\widetilde T(p, q(t)) \ge T(\tilde p, \tilde q(s))$. This means that the last parameter $t_0$ where $T(\tilde p, \tilde q(s))$ meets $\widetilde T(p, q(t))$ must be $t$. This contradicts  $T(\tilde p, \tilde q(t))\cap T(\tilde p, \tilde q(s)) = \{ \tilde p \}$ again. 
Since $\widetilde T(p,q(t)) \ge T(\tilde p, 
\tilde q(s))$, we conclude that $T(\tilde p, \tilde q(t)) \ge T(\tilde p, \tilde 
q(s))$. By Lemma \ref{TandR} (4), we have the same inequality for the reference reverse curves. This completes the proof of (1). 
 
Given $t$, $0 \le t \le d(p,q)$, we set $c(s)=T(\tilde p, \tilde q(s))(t)=S(\tilde p, t)\cap T(\tilde 
p, \tilde q(s))$ for any $s \in (t, d(p,q)]$. Then, (1) implies that $r(c(s))$ 
is monotone nonincreasing for $s > t$. We then have 
\[ 
r(\tilde q(t)) \ge r(c(s)) \ge r(c(d(p,q))) = r(T(\tilde p, \tilde q)(t)). 
\] 
In the same way, we have $r(\tilde p(t)) \ge r(T(\tilde p, \tilde q)(d(p,q)-t))$ for any $t$. This completes the proof of (2). 

In order to prove (3) we recall that 
\[ 
\cos \angle opq =\lim_{t \to +0}\frac{d(p,q(t))^2+d(o,p)^2 
-d(o,q(t))^2}{2d(p,q(t))d(o,p)}. 
\] 
Therefore, we have from (2) 
\begin{eqnarray*} 
\cos \angle opq &=& \lim_{t \to +0}\frac{t^2+r(\tilde p)^2 -r(\tilde 
q(t))^2}{2tr(\tilde p)} \\ 
&\le& \lim_{t \to +0}\frac{t^2+r(\tilde p)^2 -r(T(\tilde p,\tilde 
q)(t))^2}{2tr(\tilde p)} =\cos \angle \tilde o \tilde p \tilde q. 
\end{eqnarray*} 
Using the reference reverse curve $\widetilde R(p,q)$, we have $\angle oqp \ge \angle \tilde o \tilde q \tilde p$ in the same way.
This completes the proof of (3). 
\end{proof} 
 
\section{Reference curves meeting no cut point}
In this section, we assume that a complete pointed Riemannian manifold $(M,o)$ is referred to a surface of revolution $(\widetilde M, \tilde o)$. 
When $\ell < \infty$, it has been proved in \cite{IMS-0} that $M$ is isometric to the warped product manifold whose warping function is the radial curvature function of $\widetilde M$ if there exists a point $p \in M$ such that $d(o,p)=\ell$. 
Hence, there is nothing to study for the comparison theorems on those manifolds anymore. Therefore, we may assume that $d(o, p) < \ell$ for all points $p \in M$.
We study the global positional relation between reference curves $\widetilde T(p,q)$ and minimizing geodesic segments $T(\tilde p, \tilde q)$. We start from the following lemma, showing the local relation, which is proved in \cite{IMS-0}, \cite{IMS} and \cite{KT2}.
 
\begin{lem}\label{fundamental}
Let $p$ be a point in $M$ such that $p \not= o$.
There exists an $r_p > d(o,p)$ such that any 
geodesic triangle $\triangle opq$ in $M$ with $d(o,q) + d(p,q) < r_p$ has a 
comparison triangle $\triangle \tilde o \tilde p \tilde q$ in $\widetilde M$ 
satisfying $(2.2)$ and $(2.3)$. Moreover, if 
one of the equalities holds in $(2.2)$ and $(2.3)$, then $\triangle opq$ bounds a 
totally geodesic 2-dimensional submanifold in $M$ which is isometric to a 
comparison triangle domain $\triangle \tilde o \tilde p \tilde q$ corresponding to 
$\triangle opq$. 
\end{lem}

\begin{proof}
As was seen in Lemma \ref{ellipse} (2), the set of all ellipses $E(\tilde o, \tilde p \, ; a)$, $a > d(\tilde o, \tilde p)$, gives a foliation of $\widetilde M \setminus T(\tilde o, \tilde p)$. 
Namely, for any point $\tilde q \in \widetilde M \smallsetminus T(\tilde o, \tilde p)$ there exists the unique ellipse $E(\tilde o, \tilde p \, ; a)$ passing through $\tilde q $.
When $r(\tilde p) < \ell$, there exists a positive $\delta$ such that the $\delta$-neighborhood $D(\delta)$ of $T(\tilde o, \tilde p)$ does not contain any cut point of $\tilde p$. 
Observe that the proof of the comparison theorems in \cite{IMS-0}, \cite{IMS} and \cite{KT2} is valid if the domain is free from $Cut(\tilde p)$. Hence, if we set $r_p=\max \{ a \, | \, E(\tilde o, \tilde p \, ; a) \subset D(\delta) \}$, then it satisfies this lemma.
\end{proof}

It follows from (2.2) and the third inequality of (2.3) that the reference curves and the comparison triangle $\triangle \tilde o \tilde p \tilde q$ actually lie in $\widetilde M^+_{\tilde p}$ for all points $q$ with $d(o,q) + d(p,q) < r_p$.

\begin{cor}\label{equalityCase}
Let $p$ and $q$ be points in $M$ other than $o$. 
Assume that a minimizing geodesic segment $T(p,q)$ is contained in $F_p{}^{-1}( \widetilde F_{\tilde p}(\widetilde M_{\tilde p}^+))$. 
If $\widetilde T(p,q)=T(\tilde p, \tilde q)$ as a set, then $\triangle opq$ bounds a totally geodesic 2-dimensional submanifold in $M$ which is isometric to a comparison triangle domain $\triangle \tilde o \tilde p \tilde q$ corresponding to $\triangle opq$.
\end{cor}

\begin{proof}
As before, set $q(t)=T(p,q)(t)$ and $\tilde q(t)=\widetilde T(p,q)(t)$, $0 \le t \le d(p,q)$. 
Since $d(\tilde p, \tilde q(t))=t$, $0 \le t \le d(p,q)$, and $\widetilde T(p,q)=T(\tilde p, \tilde q)$, we have $\tilde q(t)=T(\tilde p, \tilde q)(t)$ for all t. 
In fact, if $\tilde q(t) \not= T(\tilde p, \tilde q)(t)$ for some $t \in (0,d(p,q))$, then there exists $t_0$ such that $t_0 \not= t$ and $\tilde q(t)=T(\tilde p, \tilde q)(t_0)$. 
We then have $t=d(p,q(t))=d(\tilde p, \tilde q(t))=d(\tilde p, T(\tilde p, \tilde q)(t_0))=t_0$, a contradiction. Hence, if $0 \le t < s \le d(p,q)$, we then have $\widetilde T(q(t), q(s))(s-t)=T(\tilde q(t), \tilde q(s))(s-t)$, since $d(\tilde q(t), \tilde q(s)) = s-t=d(q(t), q(s))$, $r(q(t))=r(\tilde q(t))$ and $r(q(s))=r(\tilde q(s))$. 
Hence, we have $\angle oq(t)q(s) = \angle \tilde o \tilde q(t) \tilde q(s)$. 
It follows from Lemma \ref{fundamental} that there exists a $\delta > 0$ such that if $| s - t | < \delta $, then $\triangle oq(t)q(s)$ bounds a totally geodesic 2-dimensional submanifold in $M$ which is isometric to a comparison triangle domain $\triangle \tilde o \tilde q(t) \tilde q(s)$ corresponding to $\triangle oq(t)q(s)$. 
This shows that there exist a totally geodesic 2-dimensional submanifold $\triangle$ bounded by $\triangle opq$ and  an isometry from $\triangle$ onto the domain bounded by $\triangle \tilde o \tilde p \tilde q$ in $\widetilde M$.
\end{proof}

\begin{rem}\label{curvature}{\rm 
Our proof technique to be employed in the theorems makes it complicated to treat the case where $\widetilde T(p,q)=T(\tilde p, \tilde q)$.
In order to avoid the case, we employ the same ideas developed in Chapter 2 in \cite{CE}.

Let $K(r)$, $r \in [0, \ell)$, denote the Gauss curvature of $\widetilde M$ on the parallel $r$-circle. 
For a sufficiently small $\delta >0$ we 
consider a differential equation 
\[ 
f''(r)+(K(r)-\delta)f(r)=0. 
\] 
We denote by  $f_{\delta}(r)$ its solution with $f_{\delta}(0)=0$ and 
$f_{\delta}{}'(0)=1$. Then, $f_{\delta}(r) > f(r)$ for any $r \in (0, \ell)$. 
By defining a metric to be 
\[ 
ds^2 = dr^2 + f_{\delta}(r)^2d\theta^2, 
\] 
we have a surface of revolution $\widetilde M_{\delta}$ such that $M$ is 
referred to $\widetilde M_{\delta}$. When $\ell < \infty$, the coefficient $K(r) - \delta$ 
and the solution $f_{\delta}(r)$ are extended on an interval $[0, \ell']$ 
containing $[0, \ell]$ properly and we do not assume that $f_{\delta}(\ell')=0$ 
and $f_{\delta}{}'(\ell')=-1$. To avoid the confusing case where some equality holds in (2.2) or (2.3), we employ $\widetilde M_{\delta}$ instead of $\widetilde M$. We prove our resullts by thinking of $\widetilde M_{\delta}$ as the reference surface, 
and then conclude the proof by letting $\delta\to 0$. More precisely, we choose 
$\delta=\delta (R)$ for each $R$ with $0 < (R + d(o,p))/2 < \ell$ such that (2.1) holds in the insides of $E(o,p \, ; R)$
 and $E(\tilde o, \tilde p \, ; R)$ and such that $\delta (R)$ converges to $0$ as 
 $(R+d(o,p))/2 \rightarrow \ell$. We prove (2.2) and (2.3) in the interior of $B(o,p \, ; R)$ and 
$B(\tilde o, \tilde p \, ; R)$, and then take $(R+d(o,p))/2$ to $\ell$. 
The 
most important fact is that $\widetilde T(p,q)=T(\tilde p, \tilde q)$ does not 
occur in $\widetilde M_{\delta}$ for any points $q \not= p$ in $E(o,p \, ; R)$. 
This property simplifies our discussion.
}\end{rem}

The following lemma is proved in \cite{IMS}. The proof here is 
different from theirs. Moreover, the method in the proof will be used when we 
prove Theorems in \S 7. \par
 
\begin{lem}\label{comparison}
Assume that a point $q \in M$ admits a minimizing geodesic segment $T(p,q)$ contained in $F_p{}^{-1}( \widetilde F_{\tilde p}(\widetilde M_{\tilde p}^+))$. 
If there exists a point $q_1 \in T(p,q)$ such that $\widetilde T(p,q') \ge T(\tilde p, \tilde q')$ for all $q' \in T(p,q_1)$, $\tilde q_1 \not\in Cut (\tilde p)$ and $((\widetilde T(p,q) \smallsetminus \{ \tilde q \}) \smallsetminus \widetilde T(p, q_1))\cap Cut(\tilde p)=\emptyset $, 
then there exists a minimizing geodesic segment $T(\tilde p, \tilde q)$ such that $\widetilde T(p,q) \ge T(\tilde p, \tilde q)$. In addition, if $\widetilde T(p,q)\cap T(\tilde p, \tilde q)$ contains a point $\tilde q'$ other than $\tilde p$ and $\tilde q$, then $\triangle opq$ bounds a 
totally geodesic 2-dimensional submanifold in $M$ which is isometric to a 
comparison triangle domain $\triangle \tilde o \tilde p \tilde q$ corresponding to 
$\triangle opq$ in $\widetilde M$.
\end{lem} 
 
The point is if $\tilde q_1 \not\in Cut(\tilde p)$ or not. In case of 
$\tilde q_1 \not\in Cut(\tilde p)$ the reference curve can be extended, keeping the positional relation to a minimizing geodesic connecting its end points.
 
\begin{proof}
In order to prove the first part, we work in $\widetilde M_{\delta}$ to avoid the case where a reference curve is identified with a minimizing geodesic segment connecting its endpoints.
For convenience we set $q(t)=T(p,q)(t)$ and $\tilde q(t)=\widetilde T(p,q)(t)$ 
for any $t \in (0, d(p,q))$. Let $t_0$ be the least upper bound of the set of 
all $t_2 \le d(p,q)$ so that there exists a minimizing geodesic 
segment $T(\tilde p, \tilde q(t))$ with $\widetilde T(p,q(t)) \ge T(\tilde p, 
\tilde q(t))$ for all $t \in (0, t_2)$. 
If $t_1$ is the parameter such that $q_1=q(t_1)$, we then have $t_0 \ge t_1$ because of the assumption.

Suppose for indirect proof that $t_0 < d(p,q)$. 
Since we assume that $\tilde q(t_0) \not\in Cut(\tilde p)$, there exists a 
neighborhood $V$ of $\tilde q(t_0)$ such that $T(\tilde p, \tilde x)$ is the 
unique minimizing geodesic segment connecting $\tilde p$ and $\tilde x \in V$. 
Since the minimizing geodesic segment $T(\tilde p, \tilde q(t_0))$ is unique, 
$\widetilde T(p,q(t))\ge T(\tilde p,\tilde q(t))$ for all $t \in (0, t_0)$ implies that
 $\widetilde T(p,q(t_0)) \ge T(\tilde p, \tilde q(t_0))$. 

We will prove that there exists an $\varepsilon > 0$ such that 
$\widetilde T(p, q(t_0+t)) \ge T(\tilde p, \tilde q(t_0+t))$,\; $0 \le t \le \varepsilon $. 
Suppose that there exists a monotone decreasing sequence $t_j$ converging to $0$ such that no minimizing geodesic segment $T(\tilde p, \tilde q(t_0+t_j))$ 
satisfies $\widetilde T(p,q(t_0+t_j)) \ge T(\tilde p, \tilde q(t_0+t_j))$. 
We then have either 
\[ 
\widetilde T(p, q(t_0+t_j)) \le T(\tilde p, \tilde q(t_0+t_j)) \quad\quad 
\] 
or 
\[ 
\widetilde T(p, q(t_0+t_j)) \cap T(\tilde p, \tilde q(t_0+t_j)) \not= \{ \tilde 
p, \tilde q(t_0+t_j) \}. 
\] 

Suppose the first is true. We then have $\widetilde T(p,q(t_0))=T(\tilde p, \tilde q(t_0))$. In fact, since  $T(\tilde p, \tilde q(t_0+t_j))$ converges to $T(\tilde p, \tilde q(t_0))$ which is the unique minimizing geodesic segment connecting $\tilde p$ and $\tilde q(t_0)$, we have $\widetilde T(p, q(t_0)) \le T(\tilde p, \tilde q(t_0))$. 
Combining $\widetilde T(p,q(t_0)) \ge T(\tilde p, \tilde q(t_0))$, we conclude $\widetilde T(p,q(t_0))=T(\tilde p, \tilde q(t_0))$. Since we employ $\widetilde M_{\delta}$, this yields a contradiction because of Corollary \ref{equalityCase}. 

Suppose the second is true. Let $\tilde q_j$ be a point in $\widetilde T(p,q(t_0+t_j))\cap T(\tilde p, \tilde q(t_0+t_j))$ such that it is different from $\tilde p$, $\tilde q(t_0+t_j)$
and $T(\tilde q_j, \tilde q(t_0+t_j))\not\le  \widetilde T(p,q(t_0+t_j))$.
Let $q_j \in T(p,q(t_0+t_j))$ be the point with $F_p(q_j)=\widetilde F_{\tilde p}(\tilde q_j)$. 
If $\tilde q_j$ does not converge to the point $\tilde q(t_0)$, then there exists an accumulation point $\tilde q' \not= \tilde q(t_0)$ such that $\tilde q' \in T(\tilde p, \tilde q(t_0))$.
This situation implies that $\widetilde T(p, q(t_0)) \ge T(\tilde p, \tilde q(t_0))$ and $\widetilde T(p, q(t_0)) \cap T(\tilde p, \tilde q(t_0)) \supset \{ \tilde p, \tilde q', \tilde q(t_0) \}$, which is the assumption of the second part of this lemma, to be proved in the next paragraph.
This is impossible because we now work in $\widetilde M_{\delta}$.
We have proved that $\tilde q_j$ converges to $\tilde q(t_0)$. We then have $\tilde q(t_0+t_j)$ such that
$\widetilde T(p,q_j) \ge T(\tilde p, \tilde q_j)$ and $\widetilde T(p,q(t_0+t_j))\smallsetminus\widetilde T(p,q_j) \not\ge T(\tilde q_j, \tilde q(t_0+t_j))$, since there exists the unique minimizing geodesic segment $T(p,q_j)$ which is a subsegment of $T(p,q(t_0+t_j))$. 
On the other hand, for sufficiently large $j$, it follows from Lemma \ref{TandR} (3) that $\widetilde R(p,q(t_0+t_j))$ passes through $\tilde q_j$. 
From Lemma \ref{basic} (6) the reference reverse curve $\widetilde R(q_j,q(t_0+t_j))$ is a subarc of $\widetilde R(p,q(t_0+t_j))$ from $\tilde q(t_0+t_j)$ to $\tilde q_j$ which lies in the same side as the subarc of $\widetilde T(p,q(t_0+t_j))$ from $\tilde q_j$ to $\tilde q(t_0+t_j)$ (see Lemma \ref{TandR} (4) and (5)). 
Thus, we have the positional relation
\begin{eqnarray*} 
\widetilde R(q_j,q(t_0+t_j)) \le T(\tilde q_j, \tilde q(t_0+t_j)).
\end{eqnarray*}
However, this contradicts Lemma \ref{fundamental} near the point $\tilde q(t_0+t_j)$. We conclude that $t_0=d(p,q)$ by employing $\widetilde M_{\delta}$. Letting $\delta \rightarrow 0$ we complete the proof of the first part.

We prove the second part. 
If $\widetilde T(p,q)\smallsetminus \widetilde T(p,q') \not\subset T(\tilde p, \tilde q)$, then there exists a point $q'' \in T(q', q)$ such that $\tilde q''$ does not lie in $T(\tilde p, \tilde q)$ and hence $\tilde q'' > T(\tilde p, \tilde q)$. Therefore, $\tilde q' < T(\tilde p, \tilde q'')$, contradicting that $\widetilde T(p,q'') \ge T(\tilde p, \tilde q'')$.
Thus, we have $\widetilde T(p,q)\smallsetminus \widetilde T(p,q') \subset T(\tilde p, \tilde q)$, in other words, $\widetilde R(q',q) \subset T(\tilde p, \tilde q)$.
Let $u$ and $u'$ be points in $T(p,q)$ such that they are near $q'$ and $p$, $u$, $q'$ and $u'$ lie in this order in $T(p,q)$. If $\tilde u \not \in T(\tilde p, \tilde q)$, then we have a contradiction from Lemma \ref{fundamental} and the same argument above.
Thus, the segment $T(q',q)$ satisfying $\widetilde R(q',q) \subset T(\tilde p, \tilde q)$ can be extended until $q'$ reaches $p$. Hence, we have $\widetilde R(p,q)=T(\tilde p, \tilde q)$, and, equivalently, $\widetilde T(p,q)=T(\tilde p, \tilde q)$.
It follows from Corollary \ref{equalityCase} that $\triangle opq$ bounds a totally geodesic 2-dimensional submanifold in $M$ which is isometric to a 
comparison triangle domain $\triangle \tilde o \tilde p \tilde q$ corresponding to 
$\triangle opq$ in $\widetilde M$.
\end{proof} 
  
\section{Reference curves meeting cut points}

For two points $\tilde x,\,\tilde y\in\widetilde M^+_{\tilde p}$ with $\theta(\tilde x)\neq\theta(\tilde y)$, let $U(\tilde x,\tilde y)$ and $L(\tilde x,\tilde y)$ denote the minimizing geodesic segments joining $\tilde x$ to $\tilde y$ such that
  \[   U(\tilde x,\tilde y)\ge T(\tilde x,\tilde y)\ge L(\tilde x,\tilde y),\quad\text{for all $T(\tilde x,\tilde y)$}      \]
Notice that  $U(\tilde x,\tilde y)=L(\tilde x,\tilde y)$ if and only if $\tilde y\notin Cut(\tilde x)$ or $\tilde y\in Cut(\tilde x)$ is an end point of $Cut(\tilde x)$ such that $\tilde y$ is an isolated conjugate point to $\tilde x$ along the unique minimizing geodesic.
The following lemma is a consequence of Lemma \ref{SegEllipse} and plays an important role for the proof of our Theorems.
 
\begin{lem}\label{upper}
Assume that $B(o, p \, ; r) \subset F^{-1}_p(\widetilde F_{\tilde p}(\widetilde M^+_{\tilde p}))$.
If $q\in E(o,p \, ; r)$ is not a local maximum point of $d_r$ on 
$E(o,p \, ; r)$, then there exists a sequence of points $q_j \in E(o,p \, ; r)$ 
converging to $q$ such that $T(\tilde p, \tilde q_j)$ converges to $U(\tilde p, 
\tilde q)$ as $j \rightarrow \infty$. In particular, if $U(\tilde p, \tilde q) \not= L(\tilde p, \tilde q)$, then any extension of $U(\tilde p, \tilde q)$ crosses 
$Cut(\tilde p)$ from the far side of $\tilde o$ to the near side of $\tilde o$. 
\end{lem} 
 
\begin{proof} 
Let $q_j$ be a sequence of points in $E(o,p \, ; r)$ converging to $q$ such that 
$d(o,q_j) > d(o,q)$ for all $j$. Then we have 
\begin{eqnarray*} 
d(p,q_j) &=& r - d(o,q_j) \\ 
&< & r - d(o,q) = d(p,q). 
\end{eqnarray*} 
In view of Lemma \ref{ellipse}, we observe that $T(\tilde p, \tilde q_j)$ does not cross $[\theta (\tilde q) 
\le \theta ]\cap [r = d(o,q)]$. This means that $T(\tilde p, \tilde q_j)\smallsetminus\{ 
\tilde p\} > T(\tilde p, \tilde q) \smallsetminus \{ \tilde p \}$ for every minimizing geodesic segment $T(p,q)$. Therefore, $T(\tilde p, \tilde q_j)$ converges to 
$U(\tilde p, \tilde q)$ as $j \rightarrow \infty$.
\end{proof} 
 
We observe from Lemma \ref{comparison} that $\widetilde T(p,q) \ge U(\tilde p, \tilde q)$ 
if $\widetilde T(p,q) \smallsetminus \{ \tilde q \}\cap Cut(\tilde p) =\emptyset$ and $q \in E(o, p \, ; r)$ is not a local maximum point of $d_r$.

In the proof of the following lemma, we need an orientation of the intersection points of curves and $Cut(\tilde p)$. Let $\tilde x \in Cut (\tilde p)$. A curve $c(\theta )$, $\tilde x=c(\theta_0)$, parameterized by angle coordinate $\theta $ is said to intersect $Cut(\tilde p)$ {\it positively} (resp., {\it negatively}) at a point $\tilde x=c(\theta_0)$ if there is a small neighborhood $\Omega$ around $\tilde x$ such that  $c\cap\Omega\ge Cut(\tilde p)\cap \Omega$ for $\theta \le \theta_0$, (resp., $c\cap\Omega\le Cut(\tilde p)\cap \Omega$ for $\theta \le \theta_0$). 
Intuitively, "intersecting positively" means that $c$ meets $Cut (\tilde p )$ from the far side with respect to $\tilde o$.

\begin{lem}\label{positiveCross}
Let $q \in M$ and let $T(p,q)$ be a minimizing geodesic segment. Assume that $T(p, q) \subset F^{-1}_p(\widetilde F_{\tilde p}(\widetilde M^+_{\tilde p}))$. Suppose all 
intersection points of $\widetilde T(p,q)$ and $Cut(\tilde p)$ are positive. 
Then, we have 
\[ 
\widetilde T(p, q(t)) \ge U(\tilde p, \tilde q(t)),\quad 0 \le t < d(p,q). 
\] 
Here we set $q(t)=T(p,q)(t)$, $0 \le t \le d(p,q)$. 
\end{lem} 
 
Notice again that if $q \not\in Cut(p)$, then there exists a unique minimizing geodesic segment $T(p,q)$, and hence, the reference curve $\widetilde T(p,q)$ is 
uniquely determined. However, this does not mean that the reference curve 
connecting $\tilde p$ and $\tilde q$ uniquely exists, because $F^{-1}_p( 
\widetilde F_{\tilde p}(\tilde q))$ may not be a single point. If $\tilde q \not\in 
Cut(\tilde p)$, then there exists a unique minimizing geodesic segment $T(\tilde 
p, \tilde q)$, and hence, $T(\tilde p, \tilde q)=U(\tilde p, \tilde q)=L(\tilde 
p, \tilde q)$. 
However, if $\tilde q \in Cut(\tilde p)$, then there may be many minimizing geodesic segments $T(\tilde p, \tilde q)$. So the positional relation between $\widetilde T(p,q)$ and $U(\tilde p, \tilde q)$ is unknown, in general. These facts are often used without notice.

\begin{proof}
We work in $\widetilde M_{\delta}$ instead of the reference surface $\widetilde M$. We choose $\delta $ to be sufficiently small so that $\widetilde M_{\delta}$ satisfies the assumption in this lemma.
Let $t_0$ be the least upper bound of the set of all $t_1 \in (0, d(p,q))$ 
such that $\widetilde T(p,q(t)) \ge U(\tilde p, \tilde q(t))$ for all $t \in (0, t_1)$. 
We already know that $t_0 > 0$. Suppose for indirect proof that $t_0 < 
d(p,q)$. If $\tilde q(t_0) \not\in Cut(\tilde p)$, then, from Lemma \ref{comparison}, there exists an
$\varepsilon > 0$ such that $\widetilde T(p,q(t_0+t)) \ge U(\tilde p, \tilde 
q(t_0+t))$ for all $t \in (0, \varepsilon )$. This 
contradicts the choice of $t_0$. 

Suppose $\tilde q(t_0) \in Cut(\tilde p)$. 
Since the minimizing geodesic segment $T(p, q(t_0))$ is unique and $\tilde q(t_0)$ is a positive cut point, it follows that $\widetilde T(p,q(t_0)) \ge U(\tilde 
p, \tilde q(t_0))$. We prove that there exists an $\varepsilon > 0$ such 
that $\widetilde T(p, q(t_0+t)) \le Cut(\tilde p)$ for all $t \in (0, \varepsilon )$. In fact, suppose this is not true. 
Then, there exist a sufficiently small neighborhood $\Omega$ around $\tilde q(t_0)$ and  a sequence $t_j > 
t_0$ such that $t_j$ converges to $t_0$ and $\tilde q(t_j)$ is contained in the subdomain of $\Omega$ 
bounded below by $U(\tilde p, \tilde q(t_0))\cup Cut(\tilde p)$. Let $T_j$ be a minimizing geodesic segment 
connecting $\tilde p$ and $\tilde q(t_j)$. Then, we know that $T_j\cap \widetilde T(p,q(t_j))\not= 
\{ \tilde p, \tilde q(t_j) \}$, since $T_j$ converges to $U(\tilde p, \tilde q(t_0))$ and $\widetilde T(p,q)([0,t_0]) \not= T(\tilde p, \tilde q(t_0))$. If $\tilde q(t'_j)=\widetilde T(p,q)(t'_j) \in T_j\cap \widetilde T(p,q(t_j))$ and $\tilde q (t'_j) \not\in \{ \tilde p, \tilde q(t_j) \}$, then $\tilde q(t'_j)$ converges to $\tilde q(t_0)$ and it follows 
that 
\[ 
\widetilde T(p,q)([t'_j,t_j]) \le T(\tilde q(t'_j),\tilde q(t_j)). 
\] 
 
However, this contradicts Lemma \ref{fundamental} and Lemma \ref{TandR}. Thus we see that there exists an $\varepsilon > 0$ such that $\widetilde T(p, q(t_0+t)) \le Cut(\tilde p)$ 
for all $t \in (0, \varepsilon )$, and, therefore, $\tilde q(t_0+t) \not\in 
Cut(\tilde p)$ for all small $t > 0$. 
Since $a:=d(o,q(t_0))+ d(p,q(t_0)) < d(o, q(t_0+t))+ d(p, q(t_0+t))$, the point $\tilde q(t_0+t)$ is outside $E(o, p \, ; a)$. 
Therefore, $\tilde q(t_0+t)$  is in the subdomain of $\Omega$ bounded above by $L(\tilde p, \tilde q(t_0))\cup Cut(\tilde p)$.
Since a sequence of unique minimizing geodesic segments 
$T(\tilde p, \tilde q(t_0+t))$ converges to $L(\tilde p, \tilde q)$ as $t 
\rightarrow 0$, it follows that $\widetilde T(p, q(t_0+t)) \ge U(\tilde p, 
\tilde q(t_0+t))$ for all small $t > 0$. This contradicts the choice of 
$t_0$.
\end{proof} 

%%%%%%%%%%%%%%%%%%%%%%%%%%%%%%%%%
\section{Proof of Theorems} 
%%%%%%%%%%%%%%%%%%%%%%%%%%%%%%%%%
We are ready to prove Theorems \ref{restate1} and \ref{restate2}. 
 
\begin{proof}[Proof of Theorems \ref{restate1} and \ref{restate2}] 
Let $r_0$ be the least upper bound of the set of all $r_1 > d(o,p)$ 
satisfying the following properties: 
Let $r \in (d(o,p), r_1)$ and $q \in E(o,p \, ;r)$. Then,
 
\begin{enumerate} 
\item[(C1)]  there exists a minimizing geodesic segment $T(p, q)$ such that $T(p,q)$ is contained in the set $F_p{}^{-1}(\widetilde F_{\tilde p}(\widetilde M_{\tilde p}^+))$ and $\widetilde T(p,q) \ge U(\tilde p, \tilde q) $, 
\end{enumerate}
and
\begin{enumerate}
\item[(C2)] every minimizing geodesic segment $T(p,q)$ is contained in the set $F_p{}^{-1}( \widetilde F_{\tilde p}(\widetilde M_{\tilde p}^+))$ and satisfies
$\widetilde T(p,q) \ge L(\tilde p, \tilde q) $.
\end{enumerate} 
 
As was seen in Lemma \ref{fundamental}, we have $r_0 > d(o,p)$.
Let $r$ be such that $d(o,p) < r < r_0$. 
Let $q \in E(o,p \, ;r)$, $q_1(t)=T(o,q)(t)$ and $\tilde q_1(t)=\widetilde 
T(o,q)(t)$ for any $t \in [0, d(o,q)]$. Then, we have 
\begin{eqnarray*} 
&& d(o,q_1(t))+d(p,q_1(t)) \\ 
&= & d(o,q)-d(q, q_1(t)) + d(p, q_1(t)) \\ 
&\le &  d(o,q) + d(p, q) =r < r_0 
\end{eqnarray*} 
for any $t \in (0, d(o,q))$, and hence, from the condition (C2), every $\triangle opq_1(t)$ in $M$ has a comparison triangle $\triangle \tilde o \tilde 
p \tilde q_1(t) $ in $\widetilde M_{\tilde p}^+$ satisfying (2.2). Moreover, from the condition (C1), there exists a minimizing geodesic segment $T(p, q_1(t))$ such that  $\widetilde T(p,q_1(t)) \ge U(\tilde p, \tilde q_1(t))$.
 
\begin{ass}\label{Ass1}
Let $q \in E(o, p \, ; r)$ with $d(o, p) < r <r_0$. Assume that the minimizing geodesic segments  $T(p, q)$ and $T(\tilde p, \tilde q)$ satisfy $\widetilde T(p,q) \ge T(\tilde p, \tilde q)$. Then for any minimizing geodesic segments $T(o, p)$ and $T(o, q)$ the geodesic triangle $\triangle opq = T(o, p)\cup T(p, q) \cup T(o,q)$ has the comparison triangle $\triangle \tilde o \tilde p \tilde q$ with edge $T(\tilde p, \tilde q)$ which satisfies $(2.3)$. 
\end{ass}
\begin{proof}
It follows from Lemmas \ref{baseangle} (3), \ref{fundamental} and Corollary \ref{equalityCase} that $\angle opq \ge \angle \tilde o \tilde p \tilde q $ and $\angle oqp \ge \angle \tilde o \tilde q \tilde p$ where one of two equalities holds if and only if the geodesic triangle $\triangle opq$ bounds a totally geodesic 2-dimensional submanifold in $M$ which is isometric to a comparison triangle domain  $\triangle \tilde o \tilde p \tilde q$ corresponding to $\triangle opq$ in $\widetilde M$, because $\widetilde T(p,q)=T(\tilde p, \tilde q)$. 

In order to show $\angle poq \ge \angle \tilde p \tilde o \tilde q $, we employ $\widetilde M_{\delta (r_0)}$ instead of $\widetilde M$ and prove 
that $\theta (t) := \theta (\tilde q_1(t))=\angle \tilde p \tilde o \tilde q_1(t)$ is monotone non-increasing in $t \in [0, d(\tilde o, \tilde q)]$. 
Let $g_t(s)=d(p, q_1(t+s))$ and 
$\tilde g_t(s)=d(\tilde p, T(\tilde o, \tilde q_1(t))(t+s))$ for sufficiently small 
$s >0$. 
Here, since $T(\tilde o, \tilde q_1(t))$ lies in the meridian through $\tilde q_1(t)$, we can define the point $T(\tilde o, \tilde q_1(t))(t+s)$ for any $s \in [-t, \ell - t])$.
It follows from the conditions (C1), (C2) and Lemma \ref{baseangle} that $\pi \ge \angle oq_1(t)p > \angle \tilde o \tilde q_1(t) \tilde p$ for every 
$t \in (0, d(o,q))$. 
Let $g_{t}{}_+'(0)$ denote the 
right hand derivative at $s=0$, namely, 
\[ 
g_t{}_+'(0)=\lim_{h \to 0+0}\frac{g_t(h)-g_t(0)}{h}. 
\]
If $\alpha(t)$ is the angle of $T(\tilde o, \tilde q_1(t))$ with $U(\tilde p, \tilde q_1(t))$, then the first variation formula implies that $\tilde g_t{}_+'(0)=\cos \alpha (t)$ for every $t \in (0, d(o,q))$.
Hence, we have $g_t{}_+'(0) < \tilde g_t{}_+'(0)$ because of the condition (C1) and Lemma \ref{TandR} (4). 
There exists an 
$\varepsilon > 0$ such that 
\begin{eqnarray*} 
d(\tilde p, \tilde q_1(t+s)) &=& d(p,q_1(t+s)) = g_t(s) \\
&<& \tilde g_t(s) = d(\tilde 
p, T(\tilde o, \tilde q_1(t))(t+s)) 
\end{eqnarray*} 
for all $s \in (0, \varepsilon )$. Since $r(\tilde q_1(t+s))=r(T(\tilde o, \tilde 
q_1(t))(t+s))=t+s$, $\theta (T(\tilde o, \tilde q_1(t))(t+s))= \theta (T(\tilde o, 
\tilde q_1(t))(t))=:\theta (t)$, and $\theta (t+s) := \theta (\tilde q_1(t+s))$, it 
follows from Lemma 11 (2) that $\theta (t) > \theta (t+s)$ for all $s \in (0,
\varepsilon )$. Thus, we have 
\[ 
\angle poq = \angle \tilde p \tilde o \tilde q_1(0) > \angle \tilde p \tilde o 
\tilde q_1(d(o,q)) = \angle \tilde p \tilde o \tilde q. 
\]
Thus, we have $\angle poq \ge \angle \tilde p \tilde o \tilde q$, employing the reference surface $\widetilde M$ as $\delta (r_0)$ goes to $0$.
Here, the equality holds if and only if there exists a geodesic triangle $\triangle opq$ such that it bounds a totally geodesic 2-dimensional submanifold in $M$ which is isometric to a comparison triangle domain $\triangle \tilde o \tilde p \tilde q$ corresponding to $\triangle opq$ in $\widetilde M$ with edge $L(\tilde p, \tilde q)$.
\end{proof}

\begin{ass}\label{Ass2}
If $E(o, p \, ; r_0)\not= \emptyset$, then every point $q \in E(o, p \, ; r_0)$ satisfies that the conditions {\rm (C1)} and {\rm (C2)}. In particular, $B(o, p \, ; r_0) \subset F_p^{-1}(\widetilde F_{\tilde p}(\widetilde M_{\tilde p}^+))$.
\end{ass}
\begin{proof}
We first prove that for every point 
$q \in E(o,p \, ;  r_0)$ any minimizing geodesic segment $T(p, q)$ satisfies 
$\widetilde T(p,q) \ge L(\tilde p, \tilde q)$. Suppose $\angle oqp \not= \pi$. Let $q_j \in T(p,q) \smallsetminus \{ p, q \}$ be a sequence of points converging to $q$. Then, $r(q_j) < r_0$ is satisfied. Hence, it follows from the definition of $r_0$ 
that $\widetilde T(p,q) \ge L(\tilde p, \tilde q)$
holds as the limit of $\widetilde T(p,q_j) \ge L(\tilde p, \tilde q_j)$. 

When $\angle oqp=\pi$, it is possible that there is no sequence of points $q_j$ with $r(q_j) < r_0$ such that $q_j \rightarrow q$ (see Example \ref{Cylinder}). 
In this case, there exists a cut point $p'$ (resp., $o'$) of $o$ (resp., $p$) in $T(p,q) \smallsetminus \{ p, q \}$ (resp., $T(o,q) \smallsetminus \{ o, q \}$). 
In particular, there exists the unique minimizing geodesic segment $T(p,q)$ connecting $p$ and $q$ and $T(p,q) \supset T(p,p')$. 
Since any point $q_j \in T(p,p')\smallsetminus \{ p, p' \}$ satisfies $d(o,q_j) + d(p, q_j) < r_0$, we have $\widetilde T(p, p') \ge L(\tilde p, \tilde p')$. 
We notice that $\widetilde T(p,q)$ is the union of $\widetilde T(p,p')$ and the subarc of $E(\tilde o, \tilde p \, ; r_0)$ from $\tilde p'$ to $\tilde q$. 
Therefore, we have $\widetilde T(p,q) \ge U(\tilde p, \tilde q) \ge L(\tilde p, \tilde q)$ (see Lemma \ref{SegEllipse}).

We next prove that there exists for every point 
$q \in E(o,p \, ;  r_0)$ a minimizing geodesic segment $T(p, q)$ such that 
$\widetilde T(p,q) \ge U(\tilde p, \tilde q)$. This is the condition (C1).
If $q \not\in F_p{}^{-1}( \widetilde F_{\tilde p}(Cut(\tilde p)\cap {\rm 
Int}(\widetilde M_{\tilde p}^+)))$, then it is clear that there exists a 
minimizing geodesic segment $T(p, q)$ such that $\widetilde T(p,q) \ge U(\tilde 
p, \tilde q)$, since there is the unique minimizing geodesic segment $U(\tilde 
p, \tilde q)=L(\tilde p, \tilde q)$ connecting $\tilde p$ and $\tilde q$. 

If $q \in F_p{}^{-1}( \widetilde F_{\tilde p}(Cut(\tilde p)\cap {\rm 
Int}(\widetilde M_{\tilde p}^+)))$, then $q \not\in E_p(r_0)$ follows from the 
assumption of Theorems.
The assumption (2.1) is used only at this point.
Hence, $q$ is not a local maximum point of the distance function to $o$ restricted to $E(o, p \, ; r_0)$, namely
$d_{r_0} : E(o,p\, ;r_0)\rightarrow \mathbb{R}$. Therefore, there exists a 
sequence of points $q_j \in E(o,p \, ;r_0)$ such that $q_j \rightarrow q$ 
with $d(o,q_j) > d(o,q)$. Since the sequence of minimizing geodesic segments 
$T(\tilde p, \tilde q_j)$ and the sequence of curves $\widetilde T(p,q_j)$ 
converges to $U(\tilde p, \tilde q)$ (see Lemma \ref{SegEllipse}) and a curve $\widetilde T(p,q)$, 
respectively, it follows that there exists a minimizing geodesic segment $T(p, 
q)$ such that $\widetilde T(p,q) \ge U(\tilde p, \tilde q)$.
\end{proof}

Up to this point we have proved (2.2) and (2.3) for all points $q \in M$ with $d(o,q) + d(p,q) \le r_0$.
In order to prove that $M \subset B(o, p \, ; r_0)$, we suppose $M \smallsetminus B(o, p \, ; r_0) \not= \emptyset$ and derive a contradiction.

When we employ $\widetilde M_{\delta(R)}$, $R > r_0$, as the reference surface of $M$  (see Remark \ref{curvature}) and make the same arguments as in the proofs of Assertions \ref{Ass1} and \ref{Ass2}, we have $r_0(R)$ instead of $r_0$. 
We prove $M \subset B(o, p \, ; r_0(R))$ for all $R > r_0$ which contradicts $M \smallsetminus B(o, p \, ; r_0) \not= \emptyset$ because $\delta (R) \rightarrow 0$ as $R \rightarrow \ell$.

\begin{ass}\label{Ass3}
Suppose that $M \smallsetminus B(o, p \, ; r_0(R)) \not= \emptyset$. Then, there exists an $r_1 > r_0$  such that all points $x \in M$ with $d(o,x)+d(p,x) < r_1$ belong to $F_p{}^{-1}( \widetilde F_{\tilde p}(\widetilde M_{\delta (R)\tilde p}^+))$.
\end{ass}

\begin{proof}
Since $B(o, p \, ; r_0(R))$ is compact, 
it suffices to find an open set $U$ containing $E(o, p \, ; r_0(R))$ such that every point $q \in U$ and every minimizing geodesic segment $T(p,q)$ have the reference point $\tilde q$ and the reference curve $\widetilde T(p,q)$, respectively. 
Let $q \in E(o, p \, ; r_0(R))$. As was seen in the proof of Assertion \ref{Ass1}, $\theta (\tilde u)$ is monotone non-increasing as $u$ moves from $o$ to $q$ along $T(o,q)$. We use this fact to determine the location of the reference curve $\widetilde T(o,q)$ and to study its property.

The complicated case is that the angle of $T(\tilde o, \tilde p)^\cdot (0)$ with $T(\tilde o, \tilde q)^\cdot (0)$ is $\pi $. 
Suppose that $\angle \tilde p \tilde o \tilde q=\pi$. 
It follows from Assertions \ref{Ass1} and \ref{Ass2} that $\angle poq'=\pi$ for all $q'$ whose reference point is $\tilde q'=\tilde q$. 
Hence, we have $q' = q$ if $\tilde q' =\tilde q$.
Moreover, $\angle oup = \angle \tilde o \tilde u \tilde p$ for all $u \in T(o,q)$. 
If there exists a point $u \in T(o,q)$ such that $\angle oup \not=0$, then there exists a minimizing geodesic segment $T(p,u)$ such that $\triangle oup$ bounds a totally geodesic 2-dimensional submanifold which is isometric to the comparison triangle domain $\triangle \tilde o \tilde u \tilde p$ in $\widetilde M$.
This contradicts the present curvature condition. 
Thus, we obtain $\angle oup = \angle \tilde o \tilde u \tilde p = 0$ for all $u \in T(o,q)$. Therefore, both $T=T(p,o)\cup T(o,q)$ and $\widetilde T=T(\tilde p, \tilde o)\cup T(\tilde o, \tilde q)$ are minimizing geodesic segments.

In addition to $\angle \tilde p \tilde o \tilde q=\pi$, suppose $\tilde q \in Cut (\tilde p)$. From the present curvature assumption, $\tilde q$ is not a point conjugate to $\tilde p$ along $\widetilde T$. Hence, $\tilde q$ is not an end cut point of $\tilde p$ but branch or regular.
In particular, $U(\tilde p, \tilde q)$ is different from $\widetilde T$.
Since $q \not\in Cut (o)$, Lemmas \ref{notmaximum} and \ref{upper}, there exists a minimizing geodesic segment $T(p,q)$ such that $\widetilde T(p,q) \ge U(\tilde p, \tilde q)$. 
Thus, we conclude that $q$ is a cut point of $p$, since $T(p,q)$ is different from $T$.

Let $W'$ be a neighborhood of $q$ which is foliated by minimizing geodesic segments from $o$. Then, from the present curvature assumption, there exist a neighborhood $W \subset W'$ of $q$ and an $\varepsilon > 0$ such that, choosing the appropriate geodesic triangles, $\angle oq'p - \angle \tilde o \tilde q' \tilde p >\varepsilon $ for all points $q' \in W \cap E(o, p \, ; r_0(R))$.
Using this property and the same method as in the proof of Assertion \ref{Ass1}, we can have a neighborhood $V'_q$ of $q$ such that all points in $V'_q \smallsetminus T(o,q)$ have their reference points in ${\rm Int}(\widetilde M_{\delta (R)\tilde p}^+)$.

We next suppose $\tilde q \not\in Cut (\tilde p)$, in addition to $\angle \tilde p \tilde o \tilde q=\pi$. Let $T_e(p,o)$ be the maximal minimizing geodesic from $p$ through $o$.
If $q$ lies in $T_e(p, o)$ but not the endpoint, then the minimizing geodesic segment $T(p,q)$ is unique and $T(p, q) \subset T_e(p,o)$.
Even if $q$ is the endpoint of $T_e(p,o)$, then $T_e(p,o)=T(p,o)\cup T(o,q)$ is a minimizing geodesic segment.  
Therefore, we have $\widetilde T(p,q)=T(\tilde p, \tilde o)\cup T(\tilde o, \tilde q)$ as its reference curve, because $\angle oqp=0$ and (2.3).  

Let $N$ be the normal neighborhood around $o$, namely the domain around $o$ bounded by $Cut (o)$. Obviously, $T(p, q)=T(p,o)\cup T(o,q) \subset N$.
Because $\tilde q \not \in Cut (\tilde p)$, we can have a neighborhood $\widetilde U_{\tilde q}$ of $\tilde q$ in $\widetilde M_{\delta (R)}$ such that for all $\tilde x \in \widetilde U_{\tilde q}\cap \widetilde M_{\delta (R) \tilde p}^+$, if  we write $T(\tilde p, \tilde x)(t)=(r(t), \theta (\tilde p)+\theta (t))$ for all $t \in [0, d(\tilde p, \tilde x)]$, 
then $\exp_o(r(t)(\cos \theta (t) u+ \sin \theta (t) v)) \in N$, $0 \le t \le d(\tilde p, \tilde x)$.
Here $\exp_o : T_oM \rightarrow M$ is the exponential map and $v$ is an arbitrary unit tangent vector such that $v$ is perpendicular to $u:=T(o, p)^{\cdot}(0)$ in $T_oM$.

We prove that there exists a neighborhood $V_q$ of $q$ so that $V_q \subset N$ and the reference curve $\widetilde T(p, x)$ is defined for any $x \in V_q$. 

Let $x \in N$. Let $\theta_x$ denote the angle of $T(o,x)^\cdot (0)$ with $T(o,p)^\cdot (0)$. 
Then $\theta_x$ is continuous for $x \in N$. 
Define a map $\Psi : N \rightarrow \widetilde M_{\delta (R)\tilde p}^+$ by $\Psi (x)=(d(o,x), \theta (\tilde p )+\theta_x)$. 
Since $\Psi$ is continuous, there exists a neighborhood $V'_q$ of $q$ such that $\Psi (V'_q) \subset \widetilde U_{\tilde q}\cap \widetilde M_{\delta (R)\tilde p}^+$.

We claim that all points $x \in V'_q$ have their reference points.
Let $x \in V'_q \smallsetminus T_e(o,q)$ where $T_e(o,q)$ denotes the maximal minimizing geodesic from $o$ through $q$.
Since $x \not\in T_e(o,q)$, we have $\theta_x \not= \pi$.
Let $r(t)$ and $\theta (t)$ satisfy the equation $T(\tilde p, \Psi (x))(t)=(r(t), \theta (\tilde p)+\theta (t))$, $0 \le t \le d(\tilde p, \Psi (x))$. 
Let $u:=T(o,p)^{\cdot}(0)$ and $v_x=T(o,x)^\cdot (0)$. Set $v=(v_x-\cos \theta_x u)/\sin \theta_x$ which is the unit tangent vector perpendicular to $u$ and contained in the subspace spanned by $\{ u, v_x \}$.
Then we define a curve $c(t)=\exp_o(r(t)v(t))$, $0 \le t \le d(\tilde p, \Psi (x))$, where $v(t)=\cos \theta (t) u+\sin \theta (t) v$. 
The curve $c$ connects $p$ and $x$ and its length is less than $d(\tilde p, \Psi (x))$ because of the curvature condition and the Rauch comparison theorem (see \cite{CE}). Therefore, we have $d(p,x) < d(\tilde p, \Psi (x))$.
Thus, we can define the reference point $\tilde x$ of $x$ in $\widetilde M_{\delta (R)\tilde p}^+$ because $r(\tilde x)=r(\Psi (x))=d(o,x)$, $\theta (\tilde x) \le \theta_x$ and Lemma \ref{circle} (1).
Since all points $x \in V'_q\cap T_e(o,q)$ are accumulation points of $V'_q \smallsetminus T_e(o,q)$, every point $x \in V'_q$ has its reference point $\tilde x$ in $\widetilde M_{\delta (R)\tilde p}^+$.

From $V'_q$, we can have a neighborhood $V_q$ of $q$ mentioned above.
Suppose for indirect proof that there exists a sequence of points $x_j$ converging to $q$ such that some point $y_j \in T(p,x_j)$ defines the reference curve $\widetilde T(p,y_j)$ and some point in $T(y_j, x_j)$ close to $y_j$ does not have any reference point. 
Since those points $y_j$'s satisfy $d(o,y_j)+d(p, y_j) \ge r_0$, $q \in E(o, p \, ; r_0)$ and Lemma \ref{ellipse} (1), the sequence $d(y_j, x_j)$ goes to zero, and, hence, the sequence of the points $y_j$ converges to $q$. 
Thus, $T(y_j, x_j) \subset V'_q$ for a sufficiently large $j$, contradicting that all points $x \in V'_q$ have their reference points. 
Therefore, we have the neighborhood $V_q$ as required.

Suppose that $\angle \tilde p \tilde o \tilde q < \pi$.
Let $\widetilde U_{\tilde q} \subset {\rm Int}(\widetilde M_{\delta (R)}^+)$ be a neighborhood of $\tilde q$. Then there exists a neighborhood $V'_q$ of $q$ in $M$ such that $\widetilde F^{-1}_{\tilde p}\circ F_p(V'_q)\subset \widetilde U_{\tilde q}$.
As the argument above, we can have a neighborhood $V_q$ of $q$ as required.

Thus, we have found the set $U=\bigcup_{q \in E(o,p \, ; r_0 )}V_q$ which is a neighborhood around $E(o,p \, ; r_0(R) )$ such that $U \subset F_p{}^{-1}( \widetilde F_{\tilde p}(\widetilde M_{\delta(R)\tilde p}^+))$ .
\end{proof}

\begin{ass}\label{Ass4}
There exists an $r_2$ with $r_0(R) < r_2 \le r_1$ such that the condition {\rm (C1)} is true for any point $q \in E(o,p \, ; r)$, $r_0 (R)< r < r_2$.
\end{ass}

\begin{proof}
Suppose for indirect proof that (C1) is not true for any $r > r_0(R)$,
namely there exists a sequence of $r_j > r_0(R)$ such that $r_j$ converges to $r_0(R)$ and there 
are no minimizing geodesic segments $T(p, q_j)$ with $\widetilde 
T(p,q_j) \ge U(\tilde p, \tilde q_j)$ for some $q_j \in E(o,p \, ;r_j)$. 
Suppose without loss of generality that $q_j$ converges to $q_0 \in E(o,p \, 
;r_0(R))$. 
We then have either 
\[ 
\widetilde T(p, q_j) \le U(\tilde p, \tilde q_j) \quad\quad 
\mbox{or} \quad\quad 
\widetilde T(p, q_j) \cap U(\tilde p, \tilde q_j) \not= \{ \tilde p, \tilde q_j 
\}. 
\] 
Let $q'_j = T(p,q_j)\cap E(o,p \, ;r_0(R))$. Then, we have $\widetilde T(p,q'_j) 
\subset \widetilde T(p,q_j)$, since $T(p,q'_j) \subset T(p,q_j)$. 
It follows from the choice of $r_0(R)$ and the condition (C1)
that $\widetilde T(p,q'_j) \ge U(\tilde p, \tilde q'_j)$. If the first 
inequality is true, we then have $\widetilde T(p,q_0)=U(\tilde p, \tilde q_0)$ 
as its limit. This contradicts our curvature condition $K_{\delta (R)}$. 

If the 
second situation occurs, we then have the reverse inequality for some point 
$q' \in T(p,q_j)$ near $q_j$ for sufficiently large $j$ so that 
\[ 
\widetilde R(q',q_j) \le T(\tilde q', \tilde q_j) 
\] 
because of Lemma 16 (3) and (4). 
This contradicts Lemma \ref{fundamental}. Therefore, (C1) is true for some $r_2 > r_0(R)$. 
\end{proof}

\begin{ass}\label{Ass5}
The condition {\rm (C2)} is satisfied for all $q \in E(o, p \, ; r)$, $r_0(R) < r < r_2$.
\end{ass}

\begin{proof} 
Let $q \in M$ with $d(o,q)+d(p,q) < r_2$. For convenience, we set $q(t)=T(p,q)(t)$ 
and $\tilde q(t)=\widetilde T(p,q)(t)$, $0 \le t \le d(p,q)$. Let $t_0$ 
be the least upper bound of the set of all $t_1 \le d(p,q)$ so that 
there exists a minimizing geodesic segment $T(\tilde p, \tilde q(t))$ with 
$\widetilde T(p,q(t)) \ge T(\tilde p, \tilde q(t))$ for all $t \in (0, t_1)$. Recall that $t_0 > 0$ because of Lemma \ref{fundamental}. Suppose for indirect proof that 
$t_0 < d(p,q)$. If $\tilde q(t_0) \not\in Cut(\tilde p)$, then there exists a 
positive $\varepsilon $ such that $\widetilde T(p,q(t_0+t)) \ge T(\tilde p, 
\tilde q(t_0+t))$ for all $t \in (0, \varepsilon )$ because of Lemma \ref{comparison}. This 
contradicts the choice of $t_0$. Suppose $\tilde q(t_0) \in Cut(\tilde p)$. 
Since the minimizing geodesic segment $T(p, q(t_0))$ is unique and (C1) is 
satisfied, we have $\widetilde T(p,q(t_0)) \ge U(\tilde p, \tilde 
q(t_0))$. 
As is observed in the proof of Lemmas \ref{upper} and \ref{positiveCross}, there exists a 
positive $\varepsilon $ such that $\widetilde T(p, q(t_0+t)) \ge U(\tilde p, 
\tilde q(t_0+t))$ for all $t \in (0, \varepsilon )$. This contradicts the 
choice of $t_0$. Hence, it follows that 
$\widetilde T(p,q) \ge T(\tilde p, \tilde q)$.
\end{proof}

Assertions \ref{Ass3} to \ref{Ass5} imply that $M \smallsetminus B(o, p \, ; r_0(R)) \not= \emptyset$ is false when we employ the reference surface $\widetilde M_{\delta (R)}$. 
Since $\delta (R) \rightarrow 0$ as $R \rightarrow \ell$, we conclude that $M \smallsetminus B(o, p \, ; r_0) \not= \emptyset$ is false to the original reference surface of revolution $\widetilde M$ ($\delta (\ell)=0$). This completes the proof of Theorems \ref{restate1} and \ref{restate2}. 
\end{proof} 

The following proposition has been proved in the above argument. 

\begin{prop}\label{charCut}
Let $M$ and $p$ satisfy the same assumption as in Theorem \ref{restate1}. Then, a point $q \in M$ is a cut point of $p$ if there exists a minimizing geodesic segment
$T(p,q)$ such that $\widetilde T(p,q) \not\ge U(\tilde p, \tilde q)$. 
\end{prop} 
 
\begin{proof} 
As was seen in the proof of Theorem \ref{restate1}, if the reference point $\tilde q$ is in ${\rm Int}(\widetilde M_{\tilde p}^+)$, then there exists a minimizing geodesic 
$T$ connecting $p$ and $q$ such that $\widetilde T \ge U(\tilde p, \tilde 
q)$. Therefore, we have at least two minimizing geodesics $T$ and $T(p,q)$ 
connecting $p$ and $q$. This implies that $q\in Cut(p)$.

Suppose that $\theta (\tilde q)=0$ or $\pi$. Then, as was seen in the proof of Assertion \ref{Ass3}, there are two possibilities. One is that  $\widetilde T(p,q)=T(\tilde p, \tilde o)\cup T(\tilde o, \tilde q)$ and it is a minimizing geodesic segment in $\widetilde M$. 
Then, it follows from the curvature condition that if $q \not\in Cut (p)$, then $\tilde q \not\in Cut (\tilde p)$. Then, $\widetilde T(p,q) \ge U(\tilde p, \tilde q)$ is true, a contradiction. 
The other is that a geodesic triangle $\triangle opq$ bounds a totally geodesic 2-dimensional submanifold in $M$ which is isometric to the comparison triangle domain $\triangle \tilde o \tilde p \tilde q$ in $\widetilde M$. 
In this case, for any $\delta >0$, we regard $\widetilde M_{\delta}$ as a reference surface of $\widetilde M$ and $M$. 
Then, the reference point $\tilde q \in \widetilde M_{\delta}^+$ of $q$ belongs to ${\rm Int}(\widetilde M_{\delta}^+)$ and, moreover, the boundary of the set of the reference points of all points in $\widetilde M$. If $q \not\in Cut (p)$, then this contradicts Lemma \ref{basic} (1), meaning that $q \in Cut(p)$.
\end{proof} 
 
\begin{rem}\label{changeAss}
{\rm 
From the proof of Theorems we notice that the assumption $F_p(E(p))\cap 
\widetilde F_{\tilde p}(Cut(\tilde p)\cap {\rm Int}(\widetilde M_{\tilde 
p}^+))=\emptyset$ can be replaced by the following condition: If $q \in 
F_p{}^{-1}( \widetilde F_{\tilde p}(Cut(\tilde p)\cap {\rm Int}(\widetilde 
M_{\tilde p}^+)))$, then there exists a minimizing geodesic segment $T(p,q)$ with 
$\widetilde T(p,q) \ge U(\tilde p, \tilde q)$. 
} 
\end{rem}

\section{Maximum perimeter and diameter}
We have corollaries which are the special version of Corollary \ref{App-1}.

\begin{cor}\label{circumference}
Let $(M, o)$ and all points $p \in M$ with $p \not= o$ satisfy the same assumption in Theorem \ref{restate1}.
Then evry geodesic triangle $\triangle opq$ admits its comparison triangle $\triangle \tilde o \tilde p \tilde q $ in $\widetilde M$. In particular, if $\ell < \infty$, we have
\[
d(o,p)+d(p,q)+d(o,q) \le 2\ell 
\]
and the diameter of $M$ is less than or equal to $\ell$.
\end{cor}

\begin{proof}
As was seen in the paragraph just before Lemma \ref{ellipse}, we have
\[
d(o,q)+d(p,q) =d(\tilde o, \tilde q)+d(\tilde p, \tilde q) \le 2\ell - d(\tilde p, \tilde o) =2\ell - d(p,o).
\]
Therefore, we have $d(o,p)+d(p,q)+d(o,q) \le 2\ell $.

Let $p$ and $q$ be points in $M$ such that $d(p,q)$ is the diameter of $M$. It is clear that if $p=o$, then $d(p,q) \le \ell$. Suppose that $p \not= o$. 
If $\tilde p$ and $\tilde q$ in $\widetilde M$ are the reference points of $p$ and $q$, respectively, we then have
\[
d(\tilde p, \tilde q) \le \min \{ d(\tilde p, \tilde o)+d(\tilde o, \tilde q), d(\tilde p, \tilde o_1)+ d(\tilde o_1, \tilde q) \} \le \ell,
\]
where $\tilde o_1$ is the antipodal point of $\tilde o$ in $\widetilde M$. Therefore, we have $d(p,q) \le \ell$.
\end{proof}

We have the maximum diameter theorem and the maximum perimeter theorem if the assumption (2.1) is extended to the boundary of $\widetilde M_{\tilde p}^+$.

\begin{cor}\label{maximalPerimeter}
Let $(M, o)$ be a complete pointed Riemannian manifold $(M,o)$ which is referred to a 
surface of revolution $(\widetilde M, \tilde o)$ with $\ell < \infty$. Assume that all points $p \in M$ with $p \not= o$ satisfy 
\begin{equation}
\label{cond2} 
F_p(E(p))\cap \widetilde F_{\tilde p}(Cut(\tilde p)\cap \widetilde 
M_{\tilde p}^+)=\emptyset.
\end{equation}
If there exists a pair of points $p$ and $q$ in $M$ such that the perimeter of the geodesic triangle $\triangle opq$ is $2\ell$, then $M$ is isometric to the warped product manifold whose warping function is the radial curvature function of $\widetilde M$.
In particular, the same conclusion holds for $M$ if the diameter of $M$ is $\ell$.
\end{cor}

\begin{proof}
Suppose that the perimeter of the geodesic triangle $\triangle opq$ is $2\ell$. From Theorem \ref{restate1}, it has a comparison triangle $\triangle \tilde o \tilde p \tilde q$ in $\widetilde M_{\tilde p}^+$. 
We then have $d(\tilde p, \tilde q)=d(\tilde p, \tilde o_1)+d(\tilde o_1, \tilde q)$, since
\begin{eqnarray*}
&&d(\tilde p, \tilde q)+d(\tilde p, \tilde o)+d(\tilde o, \tilde q)  \\
&=& 2\ell \\
&=& 2d(\tilde o, \tilde o_1) \\
&=&d(\tilde p, \tilde o_1)+d(\tilde o_1, \tilde q)+ d(\tilde p, \tilde o)+d(\tilde o, \tilde q),
\end{eqnarray*}
where $\tilde o_1$ is the antipodal point of $\tilde o$ in $\widetilde M$.
This implies that $U(\tilde p, \tilde q)=T(\tilde p, \tilde o_1)\cup T(\tilde o_1, \tilde q)$.
From the assumption  (\ref{cond2}) and the condition (C1), there exists a minimizing geodesic segment $T(p, q)$ in $M$ such that $\widetilde T(p,q) \ge U(\tilde p, \tilde q)$. Thus, we can find a point $o_1 \in T(p,q)$ whose reference point is $\tilde o_1$. 
Since $d(o, o_1)=d(\tilde o, \tilde o_1)=\ell$, The farthest point theorem in \cite{IMS-0} concludes our corollary.

Suppose that the diameter of $M$ is $\ell$. Let the distance between $p$ and $q$ be $\ell$. 
If $p=o$, then the statement follows from the farthest point theorem (see \cite{IMS-0}).
Suppose $p \not=o$. 
As was seen in the proof of Corollary \ref{circumference},  we have
\[
\ell = d(\tilde p, \tilde q) \le \min \{ d(\tilde p, \tilde o)+d(\tilde o, \tilde q), d(\tilde p, \tilde o_1)+ d(\tilde o_1, \tilde q) \} \le \ell.
\]
Since 
\[
d(\tilde p, \tilde o)+d(\tilde o, \tilde q)+d(\tilde p, \tilde o_1)+ d(\tilde o_1, \tilde q)=2\ell,
\]
we have 
\[
d(\tilde p, \tilde q) = d(\tilde p, \tilde o)+d(\tilde o, \tilde q) = d(\tilde p, \tilde o_1)+ d(\tilde o_1, \tilde q) = \ell.
\]
Thus, the perimeter of the comparison triangle $\triangle \tilde o \tilde p \tilde q$ in $\widetilde M$ of $\triangle opq$ is $2\ell$. The maximal perimeter theorem prove the maximal diameter theorem.
\end{proof}

\begin{rem}\label{constCurv}{\rm
If the Gauss curvature of the reference surface $\widetilde M$ is a positive constant $\kappa $, we do not need the assumption (\ref{cond2}). 
In fact, as was seen in the proof of Corollary \ref{maximalPerimeter}, we have $U(\tilde p, \tilde q)=T(\tilde p, \tilde o_1)\cup T(\tilde o_1, \tilde q)$ if the perimeter of $\triangle opq$ is $2\ell$. 
If $d(\tilde p, \tilde q) < \ell=\pi / \sqrt{\kappa}$, then the minimizing geodesic segment is unique, meaning that $U(\tilde p, \tilde q)=L(\tilde p, \tilde q)$.  
This implies that there exists a minimizing geodesic segment $T(p,q)$ in $M$ such that $\widetilde T(p,q) \ge T(\tilde p, \tilde q)$ as the limit of the positional relations in ${\rm Int}(\widetilde M_{\tilde p}^+)$.
Thus, we have a point $o_1$ whose reference point is $\tilde o_1$.
In particular, the diameter of $M$ is $\ell$. 
In the case of  $d(\tilde p, \tilde q)=\ell$, it is clear that the diameter of $M$ is $\ell$. 
Therefore, the maximum diameter theorem states that $M$ is a sphere with constant curvature $\kappa$.
}\end{rem}

\begin{proof}[Proof of Corollary \ref{App1}] 
We first prove that there exists a straight line in $\widetilde M$ if there is a 
straight line in $M$. Let $T(t)$, $-\infty < t < \infty $, be a straight line in 
$M$. Let $t_0$ be a parameter such that $d(o,T(t_0))=d(o,T)$. We set $\widetilde 
T(t_1)=(d(o,T(t_1)), 0)$ for all $t_1 \in (-\infty, t_0)$, and $\widetilde T(t) = 
\widetilde F_{\widetilde T(t_1)}{}^{-1}\circ F_{T(t_1)}(T(t))$ for any $t \in (t_1, \infty)$. Then it follows from Theorem \ref{restate1} that $\widetilde T(t)$, $t \ge t_1$, is a 
curve in $\widetilde M_{\widetilde T(t_1)}^+$ such that $\widetilde T \ge 
T(\widetilde T(t_1), \widetilde T(t))$ for all $t \ge t_1$. The sequence of 
minimizing geodesic segments $S_t=T(\widetilde T(t_1), \widetilde T(t))$ 
connecting $\widetilde T(t_1)$ and $\widetilde T(t)$ contains a subsequence 
$S_k$ converging to a ray $S$ emanating from $\widetilde T(t_1)$ as $k 
\rightarrow \infty$. Let the ray be denoted by $S(t_1)(t)$, $t_1 \le t$. Then,  $d(\tilde o, S(t_1)(t_0)) \le d(o,T)$. 
From this fact we can find a sequence of rays $S(k)$ converging to a straight 
line $S$ as $k \rightarrow -\infty $. 
 
It is known that if there is a straight line in $\widetilde M$, then the total 
curvature of $\widetilde M$ is nonpositive. Therefore $M$ has no straight line. 
If $M$ has at least two ends, then there is a straight line connecting distinct 
ends. This is impossible because the total curvature of $\widetilde M$ is 
positive. 
\end{proof}

\end{document}